\documentclass[12pt]{amsart}
\usepackage{amssymb,amsmath,amsthm,amssymb}
\usepackage{mathtools}
\usepackage{graphicx}
\usepackage{comment}
\usepackage{hyperref}
\usepackage{enumerate}
\usepackage{xcolor}
\usepackage{epstopdf}
\usepackage{enumerate}

\usepackage[margin=1.2in]{geometry}

\newcommand{\R}{\ensuremath{\mathbb{R}}}

\newcommand{\N}{\ensuremath{\mathbb{N}}}
\newcommand{\Z}{\ensuremath{\mathbb{Z}}}

\renewcommand{\epsilon}{\varepsilon}
\DeclareMathOperator{\sign}{sign}

\DeclareMathOperator{\length}{length}

\newcommand{\pd}{\partial}
\newcommand{\pr}{\partial}

\newcommand{\cd}{\nabla}

\newcommand{\inner}[2]{\left\langle #1 \, , \, #2\right\rangle} 

\def\labelitemi{--}
\def\ba #1\ea {\begin{align} #1\end{align}}
\def\bann #1\eann {\begin{align*} #1\end{align*}}
\def\ben #1\een {\begin{enumerate} #1\end{enumerate}}
\def\bi #1\ei {\begin{itemize}\renewcommand\labelitemi{--} #1\end{itemize}}

\newtheorem{theorem}{Theorem}[section]
\newtheorem{proposition}[theorem]{Proposition}
\newtheorem{lemma}[theorem]{Lemma}
\newtheorem{remark}{Remark}[section]
\newtheorem{corollary}[theorem]{Corollary}

\newtheorem{claim}{Claim}[theorem]

\newenvironment{usethmcounterof}[1]{%
  \renewcommand\thetheorem{\ref{#1}}\theorem}{\endtheorem\addtocounter{theorem}{-1}}

\newcommand{\bs}{\boldsymbol}

\title[A distance comparison principle for CSF with free boundary]{A distance comparison principle for curve shortening flow with free boundary}
\author{Mat Langford}
\address{Mathematical Sciences Institute, Australian National University, Canberra, ACT 2601, Australia}
\email{mathew.langford@anu.edu.au}
\author{Jonathan J. Zhu}
\address{Department of Mathematics, University of Washington, Seattle, WA, USA}
\email{jonozhu@uw.edu}

\begin{document}

\maketitle

\begin{abstract}
We introduce a reflected chord-arc profile for curves with orthogonal boundary condition and obtain a chord-arc estimate for embedded free boundary curve shortening flows in a convex planar domain. As a consequence, we are able to prove that any such flow either converges in infinite time to a (unique) ``critical chord'', or contracts in finite time to a ``round half-point'' on the boundary. 
\end{abstract}

\setcounter{tocdepth}{1}





\section{Introduction}

Curve shortening flow is the gradient flow of length for regular curves. It was introduced as a model for wearing processes \cite{Firey} and 
the evolution of grain boundaries in annealing metals \cite{Mullins,vonNeumann}, 
and has found a number of further applications, for example in image processing \cite{Sapiro}. It arises in areas as diverse as quantum field theory \cite{Bakas} and cellular automata \cite{MR1669736}. 
The ultimate fate of closed, embedded curves in $\mathbb{R}^2$ under curve shortening flow is characterised by the theorems of Gage--Hamilton \cite{GageHamilton86} and Grayson \cite{Grayson}, which imply that any such curve must remain embedded, eventually becoming convex, before shrinking to a round point after a finite amount of time. 


Recently, there has been significant interest in so-called \textit{free boundary} problems in geometry. Study of the free-boundary curve shortening flow (whereby the endpoints of the solution curve are constrained to move on a fixed barrier curve which they must meet orthogonally) was initiated by Huisken \cite{Huisken89}, Altschuler--Wu \cite{AltschulerWu}, and Stahl \cite{Stahl96b,Stahl96a}. In particular, Stahl proved that bounded, convex, locally uniformly convex curves with free boundary on a smooth, convex, locally uniformly convex barrier remain convex and shrink to a point on the barrier curve. 

Our main theorem completely determines the long-time behaviour of simple closed intervals under free-boundary curve shortening flow in a convex domain.

\begin{usethmcounterof}{thm:fbGrayson}
Let $\Omega \subset \mathbb{R}^2$ be a convex domain of class $C^2$ and let $\{\Gamma_t\}_{t\in [0,T)}$ be a maximal free boundary curve shortening flow starting from a properly embedded interval $\Gamma_0$ in $\Omega$. Either:
\begin{enumerate}
\item[(a)] $T=\infty$, in which case $\Gamma_t$ converges smoothly as $t\to \infty$ to a chord in $\Omega$ which meets $\pd\Omega$ orthogonally; or 
\item[(b)] $T<\infty$, in which case $\Gamma_t$ converges uniformly to some $z\in \pr\Omega$, and
\[
\tilde\Gamma_t\doteqdot \frac{\Gamma_t -z}{\sqrt{2(T-t)}}
\]
converges uniformly in the smooth topology as $t\to T$ to the unit semicircle in $T_z\Omega$.
\end{enumerate}
\end{usethmcounterof}

Theorem \ref{thm:fbGrayson} represents a free-boundary analogue of the Gage--Hamilton--Grayson theorem. Note, however, that we must allow for long-time convergence to a stationary chord, which was not a possibility for closed planar curves. Observe that our statement includes uniqueness of the limiting chord, which is a subtle issue.\footnote{Indeed, there are examples of closed curve shortening flows in three-manifolds which have non-unique limiting behaviour as $t\to \infty$ (see \cite[Remark 4.2]{WhiteNotesEdelen}). On the other hand, Gage \cite{MR1046497} showed that closed curve shortening flow on the round $S^2$ does converge to a unique limiting geodesic.  More generally, on a closed Riemannian surface, Grayson \cite{MR979601} proved that a closed curve shortening flow always subconverges to a closed geodesic as $t\to\infty$ if $T=\infty$, but uniqueness of the limiting geodesic appears to remain open.}

Returning to the setting of closed planar curves, we recall that Gage \cite{Gage84} and Gage--Hamilton \cite{GageHamilton86} established that closed convex curves remain convex and shrink to ``round'' points in finite time under curve shortening flow by exploiting monotonicity of the isoperimetric ratio and Nash entropy of the evolving curves. By carefully exploiting ``zero-counting'' arguments for parabolic equations in one space variable, Grayson \cite{Grayson} was able to show that general closed embedded curves eventually become convex. Further proofs of these results were discovered by Hamilton \cite{MR1369140}, Huisken \cite{Huisken96}, Andrews \cite{MR2967056} and Andrews--Bryan \cite{MR2843240,AndrewsBryan}. Huisken's argument provides a rather quick route to the Gage--Hamilton--Grayson theorem via distance comparison: using only the maximum principle, he shows that the ratio of extrinsic to intrinsic distances --- the \textit{chord-arc} ratio --- does not degenerate under the flow. This precludes ``collapsing'' singularity models, and the result follows by a (smooth) ``blow-up'' argument. Andrews and Bryan provided a particularly direct route to the theorem by refining Huisken's argument: they obtained a \emph{sharp} estimate for the chord-arc profile, which implied much stronger control on the evolution, allowing for a direct proof of convergence.

Inspired by the approach of Huisken and Andrews--Bryan to planar curve shortening flow, we introduce a new ``extended'' chord-arc profile for embedded curves with orthogonal contact angle in a convex planar domain $\Omega$, and show that it cannot degenerate under free boundary curve shortening flow. The latter is sufficient to rule out collapsing singularity models, which is the key step in establishing Theorem \ref{thm:fbGrayson}. 

Our extended chord-arc profile is motivated by the half-planar setting: $\Omega=\mathbb{R}^2_+$. In this case, reflection across $\pd\R_+^2$ yields a curve shortening flow of closed curves in $\mathbb{R}^2$, so any suitable notion of chord-arc profile in the free boundary setting should account for the reflected part of the curve. Accordingly, we define the ``reflected'' distance $\tilde{d}(x,y)$ to be the length of the shortest single-bounce billiard trajectory in $\Omega$ connecting $x$ to $y$, and similarly define a reflected arclength. The extended chord-arc profile (see \S \ref{sec:extended-chord-arc}) then controls the relationship between extrinsic and intrinsic distance, with and without reflection. 

Our arguments therefore fit into the broader framework of maximum principle techniques for ``multi-point'' functions. Such techniques have been successfully applied to prove a number of key results in geometric analysis (see \cite{AndrewsSurvey,MR3061135} for a survey), including the distance comparison principles of Huisken \cite{Huisken96} and Andrews--Bryan \cite{AndrewsBryan}. However, applications in the context of Neumann boundary conditions are typically much more difficult than the closed (or periodic) case.

Finally, we mention that ruling out collapsing singularity models was also a key component of work of the second author and collaborators \cite{EHIZ} on free boundary \textit{mean curvature flow}, under the assumption of mean convexity. We remark that similar techniques provide a plausible alternative route to Theorem \ref{thm:fbGrayson}, so long as a suitable ``sheeting'' theorem can be established in the absence of the convexity condition.

The remainder of this paper is organised as follows. In Section \ref{sec:prelim}, we establish some preliminaries on free-boundary curve shortening flow, and in Section \ref{sec:extended-chord-arc} we define our reflected (and extended) chord-arc profile. In Section \ref{sec:spatial variation}, we establish first and second derivative conditions on the extended chord-arc profile at a spatial minimum. We compute the time derivative of the chord-arc profile in Section \ref{sec:evolution}, and then use it, in conjunction with the spatial derivative conditions, to establish our extended chord-arc estimate. In Section \ref{sec:grayson}, we deduce Theorem \ref{thm:fbGrayson}, via two different blow-up methods (``intrinsic'' and ``extrinsic''). Finally, in Section \ref{sec:unbounded}, we discuss how our chord-arc estimates may be applied to free boundary curve shortening flows (in an unbounded convex domain) with one free boundary point and one end asymptotic to a ray.

\subsection*{Acknowledgements}
The project originated while M.L. was visiting, and while J.Z. was a research fellow at, the Australian National University. Both authors would like to thank Ben Andrews and the ANU for their generous support. 

M.L. was supported by the Australian Research Council (Grant DE200101834). J.Z. was supported in part by the Australian Research Council under grant FL150100126 and the National Science Foundation under grant DMS-1802984. 

We are grateful to Dongyeong Ko for pointing out an issue with the barrier function in an earlier version of this article.

\section{Free boundary curve shortening flow}
\label{sec:prelim}

Let $\Omega$ be closed domain in $\R^2$ with non-empty interior and $C^2$ boundary $\pd\Omega$. We shall often use the notation $S\doteqdot \pd\Omega$ and denote interiors by $\mathring{\Omega}$ and so forth. A family $\{\Gamma_t\}_{t\in I}$, of connected properly immersed curves-with-boundary $\Gamma_t$ satisfies \emph{free boundary curve shortening flow} in $\Omega$ if $\mathring{\Gamma_t}\subset\mathring{\Omega}$ and 
$\pd\Gamma_t\subset\pd\Omega$ for all $t\in I$, and there is a 1-manifold $M$ and a smooth family of immersions $\gamma:M\times I\to \Omega$ of $\Gamma_t=\gamma(M,t)$ such that
\begin{equation}\label{eq:fb-csf}
\left\{\begin{aligned}\pd_t\gamma={}&-\kappa N\;\;\text{in}\;\; \mathring{M}\times I
\\ \inner{N}{N^S}={}&0\;\;\text{on}\;\;\pd M\times I\,,
\end{aligned}\right.
\end{equation}
where $\kappa(\cdot,t)$ is the curvature of $\Gamma_t$ with respect to the unit normal field $N(\cdot, t)$ and $N^S$ is the outward unit normal to $S=\pd\Omega$ along $\gamma|_{\pd M\times I}$. We will work with the unit tangent vectors $T\doteqdot JN$ and $T^S\doteqdot JN^S$, where $J$ is the counterclockwise rotation by $\frac{\pi}{2}$ in $\mathbb{R}^2$. Up to a reparametrization, we may arrange that $T=\frac{\gamma'}{|\gamma'|}$.

We will consider the setting where $\gamma(\cdot,t)$ are embeddings (a condition which is preserved under the flow) and $\Omega$ is convex. In general, the curves $\Gamma_t$ could be either bounded or unbounded and could have zero, one or two endpoints. If $\pd\Gamma_t=\emptyset$, the work of Gage--Hamilton \cite{GageHamilton86}, Grayson \cite{Grayson} and Huisken \cite{Huisken96} provides a complete description of the flow (upon imposing mild conditions at infinity in case the $\Gamma_t$ are unbounded). Our primary interest is therefore those cases in which the $\Gamma_t$ have either one or two boundary points. 
In the latter case, the timeslices $\Gamma_t$ are compact, and solutions will always remain in some compact subset $K$ of $\Omega$. (Indeed, since $\Omega$ is convex, we can enclose the initial curve $\Gamma_0$ by suitable half-lines or chords which meet $\pd\Omega$ in acute angles with respect to the side on which $\Gamma_0$ lies; these act as barriers for the flow.)

Finally, since $\Omega$ is taken to be convex, there is no loss of generality in assuming that $\Gamma_t$ does not touch $\pr \Omega$ at interior points: the strong maximum principle ensures that interior touching cannot occur at interior times, unless $\Gamma_t=\pd\Omega$ for all $t$ and $\pd\Omega$ is flat.


\section{Extending the chord-arc profile}
\label{sec:extended-chord-arc}

Recall that the (``classical") \emph{chord-arc profile} \cite{EGF} (cf. \cite{AndrewsBryan,Huisken96}) $\psi_\Gamma$ of an embedded planar curve $\Gamma$ is defined to be
\[
\psi_\Gamma(\delta) \doteqdot \inf \{d(x,y) : x,y \in \Gamma,\ell(x,y) = \delta\}\,,
\] 
where $d(x,y	)$ is the chordlength (Euclidean distance) and $\ell(x,y)$ the arclength between the points $x$ and $y$. 


\subsection{The (reflected) profile}

For curves $\Gamma$ embedded in a convex planar domain $\Omega$ with nontrivial boundary $\pd\Gamma$ on $\pr\Omega$, we introduce an ``extended'' chord-arc profile as follows: first, we define the \emph{reflected distance} $\tilde{d}(x,y)$ between two points $x,y$ in $\Omega$  (or \emph{reflected chordlength} if $x,y\in\Gamma$) by 
\[\tilde{d}(x,y) = \min_{z\in\pr\Omega}\left(|x-z|+|y-z|\right)\] 
and the \emph{reflected arclength} $\tilde\ell(x,y)$ between two points $x,y\in \Gamma$ by 
\[
\tilde\ell(x,y)\doteqdot \min_{s\in \pd\Gamma}\big(\ell(x,s)+\ell(y,s)\big)\,.
\]
The \emph{reflected chord-arc profile} $\tilde\psi_\Gamma$ of $\Gamma$ is then defined by
\[
\tilde\psi_\Gamma(\delta) \doteqdot \inf\left\{\tilde d(x,y): x,y\in \Gamma, \tilde\ell(x,y) =\delta \right\}
\]
and the \emph{extended chord-arc profile} $\bs{\psi}_\Gamma$ is taken to be
\[
\bs{\psi}_\Gamma(\delta) \doteqdot \min\{\psi_\Gamma(\delta),\tilde{\psi}_\Gamma(\delta)\}\,.
\]

Given a parametrisation $\gamma:M\to\Omega$ of $\Gamma$, we may sometimes conflate the functions $d, \ell, \tilde{d},\tilde{\ell}$ with their pullbacks to $M\times M$ by $\gamma$. 

\subsection{The completed curve and profile}

The extended chord-arc profile of an embedded curve-with-endpoints has a natural interpretation on its formal doubling. 

Consider a connected, properly immersed curve-with-boundary $\Gamma$ in a planar set $\Omega$ with endpoints on $\pr\Omega$. Given a parametrisation $\gamma:M\to \Omega$ of $\Gamma$, we define the formal double $\bs{M} = (M\sqcup M) / \pr M$ and write $\bs{x} = (x,\sign(\bs{x}))$ for elements of $\bs{M}$, where $x\in M$ and $\sign(\bs{x})=\pm$ distinguishes to which copy of $M$ it belongs. We also define continuous curve $\bs{\gamma}: \bs{M} \to \Omega$ by $\bs{\gamma}(\bs{x})= \gamma(x)$. 


Next observe that the arclength function $\ell$ is well-defined on $\bs{M}\times\bs{M}$, and satisfies
\begin{equation}
\bs\ell(\bs{x},\bs{y})\doteqdot \left\{\begin{aligned}\ell(x,y),{}&\;\;\text{if $\sign(\bs{x}) = \sign(\bs{y})$}\\ 
\tilde{\ell}(x,y), {}&\;\;\text{if $\sign(\bs{x})\neq \sign(\bs{y})$}. 
\end{aligned}
\right.
\end{equation}
Similarly, we may define a ``completed chordlength'' function on $\bs{M}\times \bs{M}$ by
\begin{equation}
\bs{d}(\bs{x},\bs{y})\doteqdot \left\{\begin{aligned}d(\gamma(x),\gamma(y)),{}&\;\;\text{if $\sign(\bs{x}) = \sign(\bs{y})$}\\ 
\tilde{d}(\gamma(x),\gamma(y)), {}&\;\;\text{if $\sign(\bs{x})\neq \sign(\bs{y})$}. 
\end{aligned}
\right.
\end{equation}

The \emph{completed chord-arc profile} $\bs{\psi}$ of $\Gamma$ is then defined by
\[
\bs{\psi}(\delta) \doteqdot \inf\left\{\bs{d}(x,y): x,y\in \bs\Gamma, \bs{\ell}(x,y) =\delta \right\}. 
\]
Note that this coincides with the notion of extended chord-arc profile defined above.

\begin{remark}
\label{rmk:gluing}
The formal double $\bs{M}$ has an obvious smooth structure
, and the arclength $\bs{\ell}$ is smooth with respect to this structure. Furthermore, if $\Gamma$ contacts $\pr\Omega$ orthogonally, then the completed chordlength $\bs{d}$ is essentially $C^1$ on $\bs{M}$. 
This $C^1$ gluing is basically what is needed to guarantee first derivative conditions at minima of the chord-arc profile, although we have presented their proof in a more direct and precise manner; see Lemma \ref{lem:C1-endpoints} in particular. 
\end{remark}

\subsection{Variation of the (reflected) chordlength}

It will be convenient to introduce some notation (see also Figure \ref{fig:angles} below). Given $x,y\in \mathring{\Omega}$ and $z\in S=\pr\Omega$, we define the angles $\theta_{x},\theta_{y}\in \R/2\pi\Z$ by
\[
\frac{x-z}{|x-z|} = \cos\theta_{x}N^S_z + \sin\theta_{x}T^S_z 
\] 
and
\[
\frac{y-z}{|y-z|} = \cos\theta_{y}N^S_z + \sin\theta_{y}T^S_z\,.
\] 
Note that, due to our convention for $N^S$, we have $\kappa^S >0$ and $\cos\theta_{x},\cos\theta_{y} <0$. 

If $x\ne y$, then, given unit vectors $X,Y$ in $\R^2$, we may further define the angles $\alpha_x^X, \alpha_y^Y$, $\beta_x^X$, and $\beta_y^Y$ by 
\[
\frac{x-y}{|x-y|} = \cos \alpha_x^X (-JX) + \sin \alpha_x^X X = \cos \alpha_y^Y (-JY) + \sin\alpha_y^Y Y
\]
\[
\frac{x-z}{|x-z|} = \cos \beta_x^X (-JX) + \sin \beta_x^X X,
\] 
and
\[
\frac{y-z}{|y-z|} = \cos\beta_y^Y (-JY) + \sin\beta_y^Y Y. 
\] 
Note that 
\[
\langle X,Y\rangle = \cos(\alpha_x^X-\alpha_y^Y)\;\;\text{and}\;\; \inner{JX}{Y}=\sin(\alpha_x^X-\alpha_y^Y)\,,
\]
and also that
\[
\langle X,T^S_z\rangle = \cos(\beta_x^X-\theta_{x})\;\; \text{and} \;\;\langle Y,T^S_z\rangle = \cos(\beta_y^Y-\theta_{y})\,.
\] 

We emphasise that the subscripts $x$ or $y$ appear only to distinguish which vectors each angle is defined by; in particular, each of the angles defined above may depend on $x,y,z$ (and $X,Y$). 


The regularity of the distance function is well-established. Its first and second variations are given as follows \cite{AndrewsBryan,Huisken96}.

\begin{proposition}
\label{prop:d}
Denote by $\Delta = \{(x,x) : x\in \Omega\}$ the diagonal in $\Omega\times \Omega$. The distance $d$ is continuous on ${\Omega}\times {\Omega}$ and smooth on $({\Omega}\times {\Omega})\setminus \Delta$. Moreover, given $(x,y) \in ({\Omega}\times {\Omega})\setminus \Delta$ and unit vectors $X,Y$ in $\mathbb{R}^2$, we have
\[\pr^x_Xd= \inner{\frac{x-y}{d}}{X} = \sin \alpha_x^X, \]
\[\pr^y_Y d= \inner{\frac{y-x}{d}}{Y} = - \sin\alpha_y^Y , \]
\[\pr^x_X\pr^x_Xd=  \frac{1}{d}  - \frac{1}{d} \inner{\frac{x-y}{d}}{X}^2 = \frac{1}{d} \cos^2 \alpha_x^X,  \]
\[\pr^y_Y\pr^y_Yd= \frac{1}{d} - \frac{1}{d} \inner{\frac{y-x}{d}}{Y}^2 = \frac{1}{d} \cos^2 \alpha_y^Y  ,\]
\[\pr^x_X\pr^y_Yd= -\frac{1}{d}\langle X,Y\rangle - \frac{1}{d} \inner{\frac{x-y}{d}}{X}\inner{\frac{y-x}{d}}{Y}   = -\frac{1}{d} \cos \alpha_x^X \cos \alpha_y^Y\,.\]
\end{proposition}

\begin{lemma}[Snell's law]\label{lem:snell}
Given any $(x,y)\in\mathring{\Omega}\times \mathring{\Omega}$ there exists $z\in \pr \Omega$ such that
\begin{equation*}
\tilde{d}(x,y) = d(x,z) + d(y,z)\,.
\end{equation*}
The triple $(x,y,z)$ necessarily satisfies
\begin{equation*}
\theta_{x}+\theta_{y}=0\,.
\end{equation*}
\end{lemma}
\begin{proof}
Since $x$ and $y$ are interior points, the function $d(x,\cdot) + d(y,\cdot)$ is smooth on $\pr\Omega$. It thus attains its minimum over $\pd\Omega$, due to compactness of $\Omega\cap \overline B_R$ for arbitrary (large) $R$. Moreover, at any minimum $z\in\pd\Omega$, the first derivative test gives the reflected angle 
condition
\ba\label{eq:snell}
0={}&\pr^z_{T^S_z} (d(x,z) + d(y,z))\nonumber\\
={}&-\inner{\frac{x-z}{|x-z|}}{T^S_z} -\inner{\frac{y-z}{|y-z|}}{T^S_z}\nonumber\\
={}& -\sin \theta_{x} - \sin\theta_{y}.
\ea
Convexity of $\Omega$ then ensures that $\theta_{x} =- \theta_{y}$ (mod $2\pi$).
\end{proof}

\section{Spatial variation of the chord-arc profile}\label{sec:spatial variation}


For the purposes of computing spatial variations it will be convenient to restrict attention to a fixed simple closed interval $\Gamma$ in $\Omega$ which meets $\pd\Omega$ orthogonally. Throughout this section, we may assume without loss of generality that $\gamma: M\to \Omega$ is a unit speed parametrisation of $\Gamma$, and $M=[0,L]$. As in \cite[Chapter 3]{EGF}, we will control the chord-arc profile by a smooth to-be-determined function $\varphi\in C^{\infty}([0,1])$ satisfying the following properties.
\begin{enumerate}[(i)]
\item $\varphi(1-\zeta) = \varphi(\zeta)$ for all $\zeta\in[0,1]$.
\item $|\varphi'|<1$.
\item $\varphi$ is strictly concave.
\end{enumerate}

Note that, since $\varphi$ is smooth and symmetric about $\zeta=\frac{1}{2}$, the function $\varphi(\frac{\bs\ell}{\bs{L}})$ is smooth away from the diagonal $\bs D$ in $\bs M\times \bs M$. 
The following observation about such functions will be useful.

\begin{lemma}[{\cite[Lemma 3.14]{EGF}}]
Let $\varphi\in C^{\infty}([0,1])$ be any function satisfying properties (i)-(iii) above. For all $\zeta \in [0,\frac{1}{2})$, we have $\varphi'(\zeta)>0$ and $\varphi(\zeta)-\zeta \varphi' (\zeta)>0$. 
\end{lemma}

We proceed to consider the auxiliary functions on $M\times M$ given by
\[ Z(x,y) = d(\gamma(x),\gamma(y)) - \bs{L} \varphi\left(\frac{\ell(x,y)}{\bs{L}}\right), \]
\[ \tilde{Z}(x,y) = \tilde{d}(\gamma(x),\gamma(y)) - \bs{L} \varphi\left(\frac{\tilde{\ell}(x,y)}{\bs{L}}\right), \]
and the auxiliary function on $M\times M\times S$ given by
\[\bar{Z}(x,y,z) = d(\gamma(x) ,z) + d(\gamma(y) ,z) - \bs{L} \varphi\left(\frac{\tilde{\ell}(x,y)}{\bs{L}}\right).\] 
Note that $\tilde{Z}(x,y) = \min_{z\in\pd\Omega}\bar{Z}(x,y,z)$. Our completed two-point function on $\bs{M}\times \bs{M}$ is defined by 
\begin{equation}\label{eq:bs Z}
\bs{Z}(\bs{x},\bs{y})  = \bs{d}(\bs{x}, \bs{y}) - \bs{L} \varphi\left(\frac{\bs{\ell}(\bs{x},\bs{y})}{\bs{L}}\right) = \begin{cases} Z(x, y)\;\;\text{if} & \sign(\bs{x})=\sign(\bs{y}) \\ \tilde{Z}(x,y)\;\;\text{if} &  \sign(\bs{x})\neq\sign(\bs{y})\,.\end{cases}
\end{equation}

Let us define (as functions of $x,y \in M$ and $z\in S=\pd\Omega$) the angles $\alpha_x = \alpha_{\gamma(x)}^{T_x}$, $\alpha_y\doteqdot \alpha_{\gamma(y)}^{T_y}$, $\beta_x\doteqdot\beta_{\gamma(x)}^{T_x}$, $\beta_y\doteqdot\beta_{\gamma(y)}^{T_y}$, and (in a slight abuse of notation) $\theta_x\doteqdot\theta_{\gamma(x)}$ and $\theta_y\doteqdot\theta_{\gamma(y)}$. In particular, we have
\begin{equation}
\label{eq:alpha}
\frac{\gamma(x)-\gamma(y)}{|\gamma(x)-\gamma(y)|} = \cos \alpha_{x} N_{x} + \sin \alpha_{x} T_{x} = \cos \alpha_{y} N_{y} + \sin\alpha_{y} T_{y}\,,
\end{equation}
\begin{equation}
\label{eq:beta-x}
\frac{\gamma(x)-z}{|\gamma(x)-z|} = \cos \beta_{x} N_{x} + \sin \beta_{x} T_{x} = \cos \theta_{x} \, N^S_z + \sin\theta_{x} \, T^S_z
\end{equation}
and
\begin{equation}
\label{eq:beta-y}
\frac{\gamma(y)-z}{|\gamma(y)-z|} = \cos\beta_{y} N_{y} + \sin\beta_{y} T_{y} = \cos \theta_{y} \, N^S_z + \sin\theta_{y} \, T^S_z. 
\end{equation}

\begin{figure}
\centering
\includegraphics[width=0.54\textwidth]{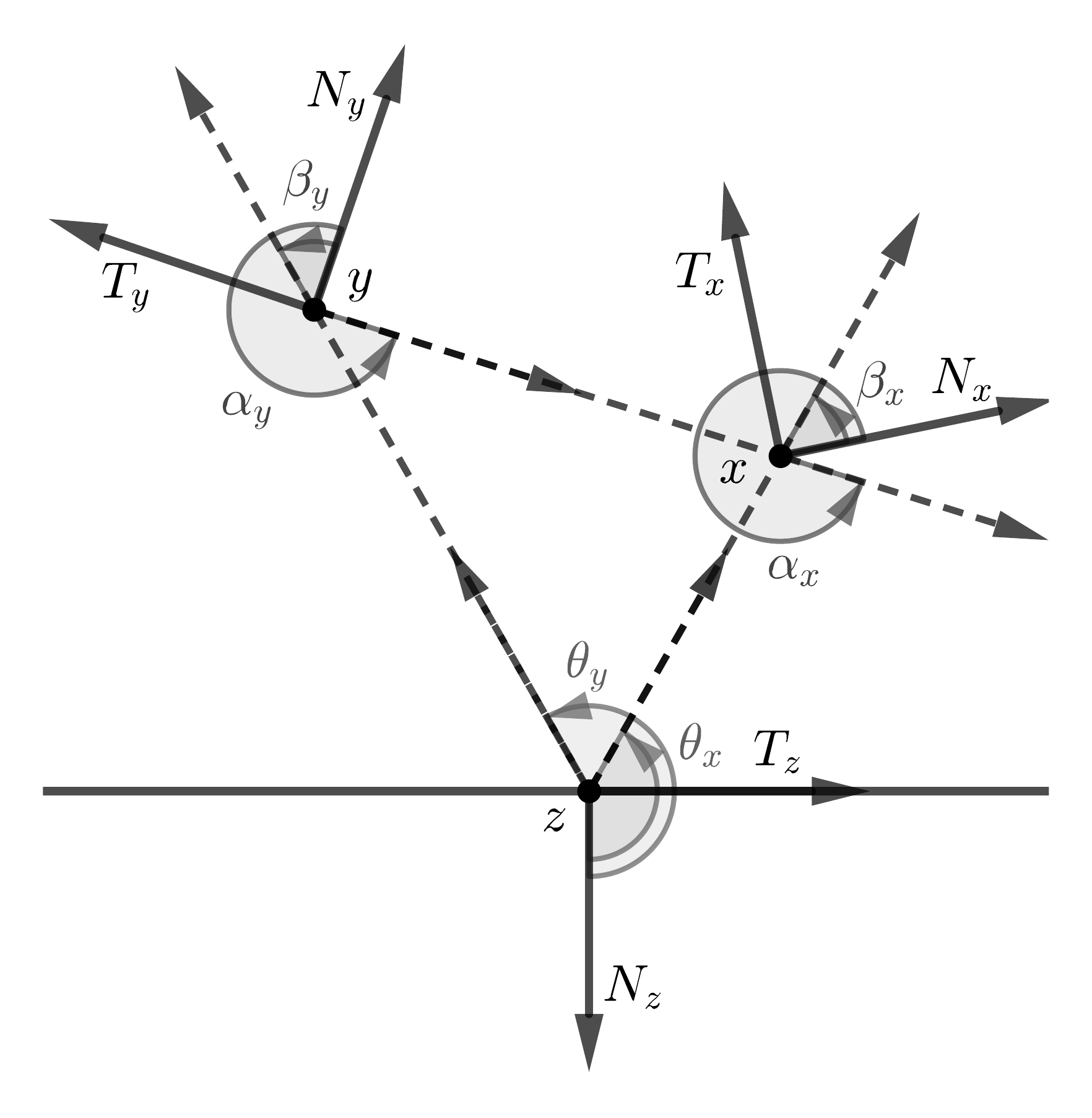}
\caption{The angles $\alpha_x,\alpha_y,\beta_x,\beta_y,\theta_x,\theta_y$, which depend on the configuration of $\gamma$ at $x,y$ and the boundary point $z$.}\label{fig:angles}
\end{figure}

\subsection{Classical profile}

We first calculate an outcome of the second derivative test at an (unreflected) minimum where the first derivative vanishes. (Note that we include the vanishing of the first derivatives as hypotheses, to account for the endpoints; these of course hold automatically at an interior minimum.) Denote by $D\doteqdot\{(x,x):x\in M\}$ the diagonal in $M\times M$.

\begin{proposition}
Suppose that $0=Z(x,y) = \min_{M\times M}Z$ and that $\pr_x Z(x,y) = \pr_y Z(x,y)=0$ for some $(x,y)\in (M\times M)\setminus D$. 
At $(x,y)$, we have
\[
\alpha_x+ \alpha_y=\pi
\] 
and 
\begin{equation}
  0\leq (\pr_x - \pr_y)^2 Z = \frac{4}{d}(1-\varphi'^2) - 4\frac{\varphi''}{\bs{L}} - \kappa_x\cos\alpha_x + \kappa_y\cos\alpha_y.
\end{equation}
\end{proposition}
\begin{proof}
Recall that $\Gamma$ is parametrised by arclength. By symmetry in $x,y$, we may also assume $x<y$, so that $\ell(x,y) = y-x$; in particular, $\pr_x\ell=-1$ and $\pr_y\ell=1$. Then $\pr_y \varphi=-\pr_x\varphi = \frac{1}{\bs{L}}\varphi'$ and hence, by Proposition \ref{prop:d}, 
\[ 0= \pr_x Z = \inner{\frac{\gamma(x)-\gamma(y)}{|\gamma(x)-\gamma(y)|}}{T_x} + \varphi' = \sin\alpha_{x} + \varphi',\]
 \[0=\pr_y Z=\inner{\frac{\gamma(y)-\gamma(x)}{|\gamma(y)-\gamma(x)|}}{T_y} - \varphi'=  -\sin\alpha_{y} - \varphi'.\]
Therefore $\sin \alpha_{x} =  \sin \alpha_{y}=-\varphi'$ and hence either $\alpha_{x}=\alpha_{y}$ or $\alpha_{x}+\alpha_{y}=\pi$. In fact, since the minimum is zero, only the latter case can occur:

\begin{claim}\label{claim:angle condition vanilla}
$\alpha_x+\alpha_y=\pi$. 
\end{claim}
\begin{proof}[Proof of Claim \ref{claim:angle condition vanilla}]
We argue as in \cite[Lemma 3.13]{EGF}. Indeed, let $\sigma$ be the line segment connecting $\gamma(x)$ to $\gamma(y)$ and let $w=\frac{\gamma(x)-\gamma(y)}{|\gamma(x)-\gamma(y)|}$. The curve $\Gamma$ divides the domain $\Omega$ into two connected components $\Omega_\pm$, where $N=N^{\Gamma}$ points towards $\Omega_+$ at all points on $\Gamma_t$. The points in $\sigma$ near $\gamma(x)$ are given by $\gamma(x)- \epsilon w$, and hence lie in $\Omega_{\sign \langle w, N_x\rangle} = \Omega_{\sign(\cos \alpha_x)}$; similarly the points in $\sigma$ near $\gamma(y)$ are given by $\gamma(y)+ \epsilon w$, and hence lie in $\Omega_{\sign \langle w, N_{y}\rangle} =\Omega_{-\sign(\cos\alpha_y)}$. If $\alpha_x=\alpha_y$, then this shows that $\sigma$ contains points on either side of $\Gamma$. In particular, $\sigma$ must intersect $\Gamma$ in a third point $\gamma(u)$. Since $d(x,u)+d(u,y)=d(x,y)$ and either $\ell(x,u)+\ell(u,y)=\ell(x,y)$ or $\ell(x,u)+\ell(u,y)=2L-\ell(x,y)$, the strict concavity of $\varphi$ now implies that either $Z(x,u)<0$ or $Z(y,u)<0$, which contradicts the assumption $\min_{M\times M}Z=0$. 
\end{proof}
 
We proceed to compute $\pr_x^2 \varphi = \pr_y^2 \varphi = -\pr_x\pr_y \varphi = \frac{1}{\bs{L}^2}\varphi''$, and so 
\[\pr_x^2 Z = \frac{1}{d}\cos^2\alpha_x - \inner{\frac{\gamma(x)-\gamma(y)}{|\gamma(x)-\gamma(y)|}}{\kappa_x N_x} - \frac{1}{\bs{L}}\varphi'' = \frac{1}{d}\cos^2\alpha_x -\kappa_x \cos\alpha_x - \frac{1}{\bs{L}}\varphi'' ,\] 
\[\pr_y^2 Z = \frac{1}{d}\cos^2\alpha_y  - \inner{\frac{\gamma(y)-\gamma(x)}{|\gamma(x)-\gamma(y)|}}{\kappa_y N_y}  - \frac{1}{\bs{L}}\varphi'' =  \frac{1}{d}\cos^2\alpha_y +\kappa_y \cos \alpha_y  - \frac{1}{\bs{L}}\varphi'',\]
and
\[\pr_x\pr_y Z = -\frac{1}{d} \cos\alpha_x\cos\alpha_y + \frac{1}{\bs{L}}\varphi''.\]
The second derivative test then gives 
\[
\begin{split}
  0 &\leq (\pr_x \pm \pr_y)^2 Z \\
  &= \frac{1}{d} ( \cos^2 \alpha_{x} + \cos^2 \alpha_{y} \mp 2\cos\alpha_{x}\cos\alpha_{y}) - (2\mp 2)\frac{\varphi''}{\bs{L}} - \kappa_x\cos\alpha_x + \kappa_y\cos\alpha_y
  \\&= \frac{1}{d} (\cos\alpha_{x} \mp \cos \alpha_{y})^2  - (2\mp 2)\frac{\varphi''}{\bs{L}}- \kappa_x\cos\alpha_x + \kappa_y\cos\alpha_y.
  \end{split}
\]

By Claim \ref{claim:angle condition vanilla} we have $(\cos \alpha_{x} +\cos \alpha_{y})^2 =  0$, and hence
\begin{equation}
0\leq (\pr_x - \pr_y)^2 Z = - 4\frac{\varphi''}{\bs{L}} - \kappa_x\cos\alpha_x + \kappa_y\cos\alpha_y\,.\qedhere
\end{equation}
\end{proof}

\subsection{Reflected profile}

We next apply the first and second derivative tests (in the viscosity sense) to the reflected profile. It will be enough to consider interior points. 

Recall that
\[\bar{Z}(x,y,z) = d(\gamma(x) ,z) + d(\gamma(y) ,z) - \bs{L} \varphi\left(\frac{\bs{\ell}(x,y)}{\tilde{L}}\right).\] 
We write $d_x = d(\gamma(x),z)$, $d_y = d(\gamma(y),z)$. 

\begin{proposition}
\label{prop:tilde-Z}
Suppose that $\bs Z\ge 0$ with $0=\tilde{Z}(x,y)$ for some off-diagonal pair $(x,y)\in (\mathring{M}\times \mathring{M})\setminus D$. At $(x,y)$, 
\begin{equation}
0\leq  -\kappa_x \cos \beta - \kappa_y \cos\beta + \left(\frac{1}{d_x} + \frac{1}{d_y}\right)\frac{2\kappa^S_z}{\left(\frac{1}{d_x} + \frac{1}{d_y}\right) \cos\theta  + 2\kappa^S_z}(1-\varphi'^2) - 4 \frac{\varphi''}{\bs{L}} \,,
\end{equation}
where $\theta = \theta_{x} = - \theta_{y}$, $\beta = \beta_{x}=\beta_{y}$ and \[\left(\frac{1}{d_x} + \frac{1}{d_y}\right) \cos\theta  + 2\kappa^S_z <0.\]
\end{proposition}
\begin{proof}
Recall that $\Gamma$ is parametrised by arclength. First, note that $\tilde{\ell}(x,y) = \min\{x+y, \bs{L}-(x+y)\}$. By reversing the parametrisation if needed, we may assume without loss of generality that $\tilde{\ell}(x,y) = x+y$, and in particular $\pr_x \tilde{\ell} = \pr_y \tilde{\ell} =1$. 

By Lemma \ref{lem:snell} there exists $z\in S$ such that $0=\tilde{Z}(x,y)= \bar{Z}(x,y,z) = \min \bar{Z}$. Moreover, we have $\theta_{x} = -\theta_{y} =: \theta$. 
Now as $x,y,z$ are all pairwise distinct, $\bar{Z}(x,y,z)$ is smooth, and we may freely apply the first and second derivative tests. 

We now have $\pr_x \varphi=\pr_y\varphi = \frac{1}{\bs{L}}\varphi'$, so the first derivatives are 
\[ \pr_x \bar{Z} = \pr^x_{T_x} d|_{\gamma(x), z} - \varphi'  = \inner{ \frac{\gamma(x)-z}{|\gamma(x)-z|}}{T_x} -\varphi' = \sin\beta_{x} - \varphi', \]
\[ \pr_y \bar{Z} = \pr^y_{T_y} d|_{\gamma(y), z} - \varphi' = \inner{ \frac{\gamma(y)-z}{|\gamma(y)-z|}}{T_y} -\varphi' = \sin\beta_{y} - \varphi', \]
and, as in Lemma \ref{lem:snell}, 
\[\pr_z \bar{Z} = \inner{\frac{z-\gamma(x)}{|z-\gamma(x)|}}{T^S_z}+\inner{\frac{z-\gamma(y)}{|z-\gamma(y)|}}{T^S_z} 
= -\sin \theta_{x} - \sin\theta_{y} =0 .\]

The first derivative test also gives $\pr_x \bar{Z}=\pr_y \bar{Z}=0$, so $\sin\beta_{x} = \sin\beta_{y} = \varphi'$. Thus, either $\beta_{x} = \beta_{y}$ or $\beta_{x} + \beta_{y}=\pi$. In fact, since the minimum is zero, only the former occurs:

\begin{claim}\label{claim:angle condition reflected}
$\beta_{x}=\beta_{y}$.
\end{claim}
\begin{figure}
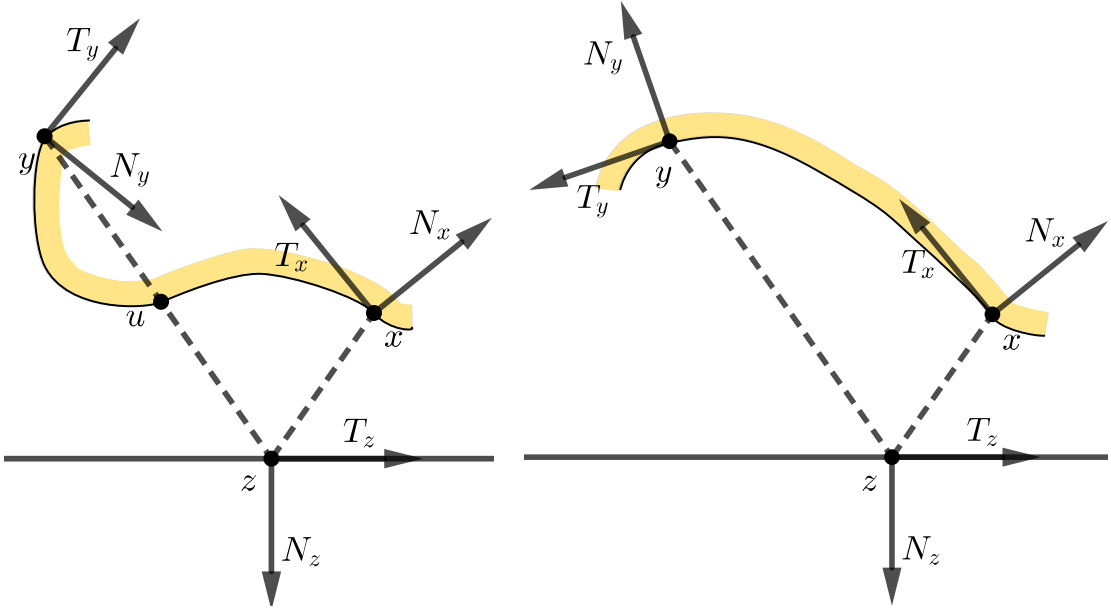

\centering
\includegraphics[width=0.42\textwidth]{wrong_config}\quad\includegraphics[width=0.5\textwidth]{right_config}
\caption{L: inadmissible configuration. R: admissible configuration.}\label{fig:reflected_configs}
\end{figure}
\begin{proof}[Proof of Claim \ref{claim:angle condition reflected}]
Let $\sigma_x$ be the line segment connecting $\gamma(x)$ to $z$ and $\sigma_y$ the line segment connecting $\gamma(y)$ to $z$, and set $w_x = \frac{\gamma(x)-z}{|\gamma(x)-z|}$ and $w_y= \frac{\gamma(y) -z}{|\gamma(y)-z|}$. Again, the curve $\Gamma$ divides the domain $\Omega$ into two connected components $\Omega_\pm$, where $N$ points towards $\Omega_+$ at all points on $\Gamma$. The points in $\sigma_x$ near $\gamma(x)$ are given by $\gamma(x)- \epsilon w_x$, and hence lie in $\Omega_{\sign \langle w, N_x\rangle} = \Omega_{\sign(\cos \beta_{x})}$; similarly the points in $\sigma_y$ near $\gamma(y)$ are given by $\gamma(y)- \epsilon w_y$, and hence lie in $\Omega_{\sign(\cos\beta_{y})}$. If $\beta_{x}+\beta_{y}=\pi$, then this shows that $\sigma_x\cup \sigma_y$ contains points on either side of $\Gamma$. In particular, $\sigma_x\cup \sigma_y$ must intersect $\Gamma$ in a third point $\gamma(u)$, $u\geq 0$ (see Figure \ref{fig:reflected_configs}). We have the following two possibilities:
\begin{enumerate}
\item $\tilde d(x,y)= d(x,u)+\tilde d(u,y)$ and either $\ell(x,u)+\tilde\ell(u,y)=\tilde\ell(x,y)$ or $\tilde\ell(x,u)+\ell(u,y)=2L-\tilde\ell(x,y)$;
\item $\tilde d(x,y)=\tilde d(x,u)+d(u,y)$ and either $\ell(x,u)+\tilde\ell(u,y)=\tilde\ell(x,y)$ or $\tilde\ell(x,u)+\ell(u,y)=2L-\tilde\ell(x,y)$. 
\end{enumerate}
So strict concavity and symmetry of $\varphi$ ensure in case (1) that either $Z(x,u)<0$ or $\tilde Z(x,u)<0$ and in case (2) that either $Z(u,y)<0$ or $\tilde Z(u,y)<0$, all of which are impossible since, by  assumption, $\bs Z\ge 0$.
\end{proof}

Henceforth, we write $\beta\doteqdot \beta_x=\beta_y$. We now compute $\pr_x^2 \varphi = \pr_y^2 \varphi = \pr_x\pr_y \varphi = \frac{1}{\bs{L}^2}\varphi''$, so the second derivatives are given by
\[
\begin{split}
\pr_x^2 \bar{Z} &= \frac{1}{d_x} - \frac{1}{d_x} \inner{ \frac{\gamma(x)-z}{|\gamma(x)-z|}}{T_x}^2  - \kappa_x \inner{ \frac{\gamma(x)-z}{|\gamma(x)-z|}}{N_x} - \frac{1}{\bs{L}}\varphi''
\\&=  \frac{1}{d_x} \cos^2\beta - \kappa_x \cos \beta - \frac{1}{\bs{L}}\varphi'', 
\end{split} \]

\[
\begin{split}
\pr_y^2 \bar{Z} &= \frac{1}{d_y} - \frac{1}{d_y} \inner{ \frac{\gamma(y)-z}{|\gamma(y)-z|}}{T_y}^2  - \kappa_y \inner{ \frac{\gamma(y)-z}{|\gamma(y)-z|}}{N_y} - \frac{1}{\bs{L}}\varphi''
\\&=  \frac{1}{d_y} \cos^2\beta - \kappa_y \cos \beta - \frac{1}{\bs{L}}\varphi'', 
\end{split} \]

\[
\pr_x \pr_y \bar{Z} =  - \frac{1}{\bs{L}}\varphi'',
\]

\begin{align*}
\pr_z^2 \bar{Z} ={}& \frac{1}{d_x} - \frac{1}{d_x} \inner{\frac{\gamma(x)-z}{|\gamma(x)-z|}}{T^S_z}^2+\kappa^S_z \inner{\frac{\gamma(x)-z}{|\gamma(x)-z|}}{N^S_z}\\&+ \frac{1}{d_y} - \frac{1}{d_y} \inner{\frac{\gamma(y)-z}{|\gamma(y)-z|}}{T^S_z}^2+\kappa^S_z\inner{\frac{\gamma(y)-z}{|\gamma(y)-z|}}{N^S_z}
\\
={}& \left(\frac{1}{d_x} + \frac{1}{d_y}\right) \cos^2\theta  + 2\kappa^S_z\cos\theta\,,
\end{align*}

\[
\begin{split}
\pr_z \pr_x \bar{Z}={}& -\frac{1}{d_x} \inner{T^S_z}{T_x} +  \inner{ \frac{\gamma(x)-z}{|\gamma(x)-z|}}{T_x}  \inner{ \frac{\gamma(x)-z}{|\gamma(x)-z|}}{T^S_z}
\\
={}& -\cos(\beta-\theta)  + \sin \beta \sin \theta\\
 ={}&-\cos\beta\cos \theta\,,
\end{split} \]
and
\[
\begin{split}
\pr_z \pr_y \bar{Z}={}& -\frac{1}{d_y} \inner{T^S_z}{T_y} +  \inner{\frac{\gamma(y)-z}{|\gamma(y)-z|}}{T_y} \inner{\frac{\gamma(y)-z}{|\gamma(y)-z|}}{T^S_z}\\
= {}&-\cos(\beta+\theta) - \sin\beta\sin \theta\\
={}& -\cos\beta\cos\theta\,.
\end{split}
\]
The second derivative test then gives
\[
0\le\pr_z^2 \bar{Z} = \left(\frac{1}{d_x} + \frac{1}{d_y}\right) \cos^2\theta  + 2\kappa^S_z\cos\theta
\]
and, for any $c$, 
\begin{align}\label{eq:second variation Z}
\qquad\qquad 0 \leq{}& (\pr_x + \pr_y + c \pr_z)^2 \bar{Z}\nonumber\\
={}& -\kappa_x \cos \beta - \kappa_y \cos\beta + \frac{1}{d_x} \cos^2 \beta + \frac{1}{d_y} \cos^2 \beta - 4\frac{\varphi''}{\bs{L}}\nonumber\\
& -2c\frac{1}{d_x}\cos \theta \cos \beta  \mp 2c\frac{1}{d_y}\cos\theta \cos \beta + c^2 \cos\theta\left(\left(\frac{1}{d_x} + \frac{1}{d_y}\right) \cos\theta  + 2\kappa^S_z\right)\\
={}& -\kappa_x \cos \beta - \kappa_y \cos\beta + \left(\frac{1}{d_x} + \frac{1}{d_y}\right)(1-\varphi'^2) - 4 \frac{\varphi''}{\bs{L}}\nonumber\\
{}&+\cos\theta \left( c^2 \left(\left(\frac{1}{d_x} + \frac{1}{d_y}\right) \cos\theta  + 2\kappa^S_z\right) - 2c\left(\frac{1}{d_x} + \frac{1}{d_y}\right)\cos\beta\right). \nonumber
\end{align}
We may actually now conclude the strict inequality
\[
\left(\frac{1}{d_x} + \frac{1}{d_y}\right)\cos\theta+2\kappa^S_z<0\,.
\]
Indeed, if the coefficient of $c^2$ were to vanish in \eqref{eq:second variation Z}, then the right hand side would be linear in $c$. Since this linear function would be bounded from below, the coefficient of $c$ would also have to vanish; i.e. $\cos\beta=0$, hence $|\varphi'| = |{\sin\beta}| =1$. By the assumptions on $\varphi$, this is only possible if $\tilde{\ell}(x,y)=0$, which in turn can only hold if $x,y$ are the same endpoint, which is not the case by hypothesis. Thus the coefficient of $c^2$ is strictly negative as claimed. 

We now take the optimal value for $c$, which is
\[
c=\frac{\left(\frac{1}{d_x} + \frac{1}{d_y}\right)\cos\beta}{\left(\frac{1}{d_x} + \frac{1}{d_y}\right) \cos\theta  + 2\kappa^S_z}\,.
\]
This yields
\[
\begin{split}
0 \leq {}& (\pr_x + \pr_y + c \pr_z)^2 \bar{Z} \\
= {}&-\kappa_x \cos \beta - \kappa_y \cos\beta + \left(\frac{1}{d_x} + \frac{1}{d_y}\right)(1-\varphi'^2) - 4 \frac{\varphi''}{\bs{L}}\\
& -\frac{\cos\theta}{\left(\frac{1}{d_x} + \frac{1}{d_y}\right) \cos\theta  + 2\kappa^S_z}\left(\frac{1}{d_x} + \frac{1}{d_y}\right)^2 \cos^2 \beta \,.
\end{split}
\]
Finally, we recall that $\cos^2 \beta = 1- \sin^2\beta = 1-\phi'^2$, which completes the proof. 
\end{proof}

\subsection{Completed profile}

Here we consider the completed two-point function $\bs Z$, which controls the completed chord-arc profile. We use the glued function to ensure that the first derivatives vanish, even at a `boundary' minimum. 

Recall that we write $\bs{x}=(x, \sign(\bs{x}))$ for elements of $\bs{M} = (M\sqcup M)/\pr M$. Also note that $\bs{Z}$ has the symmetry $\bs{Z}(\bs{x}, \bs{y}) = \bs{Z}(-\bs{x}, -\bs{y})$, where $-\bs{x} = (x, -\sign(\bs{x}))$. 

\begin{lemma}
\label{lem:C1-endpoints}
If $0=\bs{Z}(\bs{x},\bs{y}) = \min_{\bs M\times \bs M}\bs{Z}$ with $\bs{x}\in \pr M$, then \[0=Z(x,y) = \min_{M\times M} Z,\] and, moreover, $\pr_x Z|_{x,y}  = \pr_y Z|_{x,y}=0$. 
\end{lemma}
\begin{proof}
By reparametristing, we may assume without loss of generality that $M=[0, L]$, $x=0$ and $\pr_x \ell (x,y)=-1$, $\pr_y\tilde{\ell}(0,y)=1$. As $\bs{x}\in \pr M$, we have $\bs{x} = -\bs{x}$ in $\bs{M}$; by the symmetry mentioned above we have $ 0= \bs{Z}(\bs{x}, \bs{y}) = \bs{Z}(\bs{x}, - \bs{y})$. 

In particular, $0=\bs{Z}(\bs{x},(y,+))=Z(0,y) = \min_{M\times M}Z$. We will first show that $\pr_x Z|_{0,y}=0$. Since $Z$ is smooth, the first derivative test gives 

\begin{equation}
\label{eq:endpt-Z}
 0 \leq \pr_x Z(0 ,y) = \sin \alpha_x + \varphi'.
\end{equation}

On the other hand, we also have $0=\bs{Z}(\bs{x},(y,-)) = \tilde{Z}(0,y) = \min_{M\times M}\tilde{Z}(\cdot,\cdot)$. Take a unit speed parametrisation $\zeta$ of $S$ so that $\zeta(0)=\gamma(0)=:z_0$ and $\zeta'(0) = T^S_{z_0}=-N_{z_0}$ (for the last equality we have used the orthogonal contact). Then for any $c\in \mathbb{R}$ and $s\geq 0$, we must have $0\leq \bar{Z}( s, y, \zeta(cs))$, with equality at $s=0$. 
Taking the difference quotients directly gives
\[
\begin{split}
 0&\leq  \lim_{s\to 0^+} \frac{\bar{Z}(s,y, \zeta(cs)) - \bar{Z}(0,y, z_0)}{s} \\&=  \lim_{s\to 0^+} \frac{ d(s, \zeta(cs))+ d( \zeta(cs),\gamma(y)) -d( z_0, \gamma(y))}{s}   - \varphi' 
 \\&= \lim_{s\to 0^+} \frac{ d(s, \zeta(cs))}{s} -c\, (\pr^x_{N_{z_0}} d)|_{z_0, \gamma(y)} -\varphi'
  \\&= \lim_{s\to 0^+} \sqrt{1+c^2} -c \cos\alpha_x -\varphi'\,.
 \end{split}
 \]
 
 Choosing $c=\pm \cot\alpha_x$, we obtain (note that $\sin\alpha_x \leq 0$)
 \[
 0\leq 
 -\sin\alpha_x + \varphi'
 \]
Combining this with (\ref{eq:endpt-Z}), we find that indeed \[\pr_x Z(0,y) = \sin \alpha_x - \varphi'=0\] as desired. 
 
 If $\bs{y}$ is also in $\pr M$, then the same argument shows that $\pr_y Z|_{x,y} =0$. On the other hand, if $\bs{y}\notin \pr M$, then $y\in \mathring{M}$ is an interior point, and the first derivative test for $Z$ yields $\pr_y Z|_{x,y}=0$. 
\end{proof}

Lemma \ref{lem:C1-endpoints} ensures that the first derivatives vanish if a minimum occurs at an endpoint. Morally, this works because the reflected profile 
glues with the vanilla chord-arc profile in an essentially $C^1$ manner (as emphasized in Remark \ref{rmk:gluing}). 

We briefly list the remaining possibilities for ``interior'' minima:

\begin{lemma}
Suppose that $0=\bs{Z}(\bs{x},\bs{y}) = \min_{\bs M\times \bs M}\bs{Z}$, where $\bs{x}\ne \bs{y}$ and $\bs{x},\bs{y}\notin\pr M$. We may arrange that either:
\begin{enumerate}[(a)]
\item $\sign(\bs{x}) = \sign(\bs{y})$, $0=Z(x,y) =\min_{M\times M}Z$ and $\pr_x Z|_{x,y} = \pr_y Z|_{x,y}=0$; or
\item $\sign(\bs{x}) \neq \sign(\bs{y})$, and $0=\tilde{Z}(x,y) = \min_{M\times M}\tilde{Z}$. 
\end{enumerate}
\end{lemma}

Note that in the first case, the vanishing derivatives follow from the first derivative test as $Z$ is smooth ($x\neq y$).

Combining these lemmata with the second derivative tests earlier in this section yields the following dichotomy.

\begin{proposition}
\label{prop:spatial-minimum}
If $0=\min_{\bs M\times \bs M}\bs{Z}=\bs Z(\bs{x}_0,bs{y}_0)$ for some $(\bs{x}_0,\bs{y}_0)\in(\bs M\times \bs M)\setminus\bs D$, then there exist $(\bs{x},\bs{y})\in (\bs M\times\bs M)\setminus\bs D$ such that $\bs{Z}(\bs{x},\bs{y})=0$ and either:
\begin{enumerate}[(a)]
\item $\sign(\bs{x}) = \sign(\bs{y})$, $\alpha_x+ \alpha_y=\pi$, and \[ 0\leq  - 4\frac{\varphi''}{\bs{L}} - \kappa_x\cos\alpha_x + \kappa_y\cos\alpha_y; \] or
\item $\sign(\bs{x}) \neq \sign(\bs{y})$, $x,y\in\mathring{M}$ and for any $z\in S$ such that \[\tilde{d}(\gamma(x),\gamma(y)) =d_x+d_y, \qquad d_x= d(\gamma(x),z) , d_y=d(\gamma(y),z),\] we have 
\[ 0\leq -\kappa_x \cos \beta  - \kappa_{y} \cos \beta + \left(\frac{1}{d_x} + \frac{1}{d_y}\right)\frac{2\kappa^S_z}{\left(\frac{1}{d_x} + \frac{1}{d_y}\right) \cos\theta  + 2\kappa^S_z}(1-\varphi'^2) - 4 \frac{\varphi''}{\bs{L}} ,\]
where $\theta = \theta_{x} = - \theta_{y}$, $\beta=\beta_x=\beta_{y}$ and \[\left(\frac{1}{d_x} + \frac{1}{d_y}\right) \cos\theta  + 2\kappa^S_z <0\,.\]
\end{enumerate}
\end{proposition}

\section{Evolution and lower bounds for the chord-arc profile}

\subsection{Evolution of the chord-arc profile}\label{sec:evolution}

We now consider a free boundary curve shortening flow $\{\Gamma_t\}_{t\in [0,T)}$ with parametrisation $\gamma:M\times[0,T)\to \Omega$ and a smooth function $\varphi:[0,1]\times[0,T)\to\R$ satisfying the following conditions at every time $t$. (Here primes indicate spatial derivatives.)
\begin{enumerate}[(i)]
\item $\varphi(1-\zeta, t) = \varphi(\zeta,t)$ for all $\zeta\in[0,1]$.
\item $|\varphi'(\cdot, t)|<1$.
\item $\varphi(\cdot,t)$ is strictly concave.
\end{enumerate}

Denote by $\bs{d}(\cdot,\cdot,t)$, $\bs\ell(\cdot,\cdot,t)$, and $\bs{L}(t)$ the chordlength, arclength and length of the timeslice $\Gamma_t$. We consider the time-dependent auxiliary functions
\[ Z(x,y,t) = d(\gamma(x,t),\gamma(y,t)) - \bs{L}(t) \varphi\left(\frac{\ell(x,y,t)}{\bs{L}(t)},t\right), \]
\[ \tilde{Z}(x,y,t) = \tilde{d}(\gamma(x,t),\gamma(y,t),t) - \bs{L}(t) \varphi\left(\frac{\tilde{\ell}(x,y,t)}{\bs{L}(t)},t\right), \]
\[
\bs{Z}(x,y,t) \doteqdot \bs{d}(x,y,t)-\bs{L}(t)\varphi\left(\frac{\bs{\ell}(x,y,t)}{\bs{L}(t)},t\right),
\]
on $M\times M$ and 
\[\bar{Z}(x,y,z,t) = d(\gamma(x,t) ,z) + d(\gamma(y,t) ,z) - \bs{L}(t) \varphi\left(\frac{\tilde{\ell}(x,y,t )}{\bs{L}(t)},t\right)\]
on $M\times M\times S$. Denote by $[\bs{x}:\bs{y}]$ the shorter portion of $\bs{M}\setminus \{\bs{x},\bs{y}\}$. 

\begin{proposition}\label{prop:comparison equation}
Suppose that $\bs Z(\cdot,\cdot,0)\ge 0$ with strict inequality away from the diagonal. Further suppose that $t_0\doteqdot \sup\{t\in [0,T):Z(\cdot,\cdot,t)\ge 0\}<T$. Then there exist $\bs{x},\bs{y}\in (\bs M\times\bs M)\setminus \bs D$ such that $\bs{Z}(\bs{x},\bs{y},t_0)=0$ and either:
\begin{enumerate}[(a)]
\item $\sign(\bs{x}) = \sign(\bs{y})$, $\alpha_{x}+ \alpha_{y}=\pi$, and
\begin{equation}\label{eq:case-a}
0 \geq 4\frac{\varphi''}{\bs{L}}  +2 \left( \varphi - \varphi' \frac{\bs\ell}{\bs{L}}\right)\int_{\Gamma_{t}} \kappa^2 ds + \varphi' \int_{[x:y]}\kappa^2 ds- \bs{L} \pr_t \varphi
\end{equation}
or
\item $\sign(\bs{x}) \neq \sign(\bs{y})$, $x,y\in\mathring{M}$, $\beta_{x}=\beta_{y}$, and for any $z\in S$ such that \[\tilde{d}(\gamma(x),\gamma(y)) = d(\gamma(x),z) + d(\gamma(y),z) = d_x+d_y,\] we have 
\begin{equation}\label{eq:case-b}
\begin{split}
0 \geq {}& 4\frac{\varphi''}{\bs{L}}+2 \left( \varphi - \varphi' \frac{\bs\ell}{\bs{L}}\right)\int_{\Gamma_{t_0}} \kappa^2 ds + \varphi' \int_{[\bs{x}:\bs{y}]}\kappa^2 ds- \bs{L} \pr_t \varphi\\
{}&-\left(\frac{1}{d_x} + \frac{1}{d_y}\right)\frac{2\kappa^S_z}{\left(\frac{1}{d_x} + \frac{1}{d_y}\right) \cos\theta  + 2\kappa^S_z}(1-\varphi'^2) \,,
\end{split}
\end{equation}
where $\theta = \theta_{x} = - \theta_{y}$ and \[\left(\frac{1}{d_x} + \frac{1}{d_y}\right) \cos\theta  + 2\kappa^S_z <0.\]
\end{enumerate}
\end{proposition}
\begin{proof}
We will simply write $\Gamma =\Gamma_{t_0}$, etc., and we may reparametrise so that $\gamma=\gamma(\cdot,t_0)$ has unit speed. 

First observe that
\[
\lim_{s\to 0^+}\pd_x\bs Z|_{(\xi+s,\xi,\cdot)}=1-\varphi'(0,\cdot)>0 \;\;\text{and}\;\; \lim_{s\to 0^-}\pd_x\bs Z|_{(\xi+s,\xi,\cdot)}=-1+\varphi'(0,\cdot)<0
\]
and, similarly,
\[
\lim_{s\to 0^+}\pd_y\bs Z|_{(\eta,\eta+s,\cdot)}=1-\varphi'(0,\cdot)>0 \;\;\text{and}\;\; \lim_{s\to 0^-}\pd_y\bs Z|_{(\eta,\eta+s,\cdot)}=-1+\varphi'(0,\cdot)<0\,.
\]
This ensures that the diagonal is a \emph{strict} local minimum for $Z$. In fact, due to compactness of $[0,t_0]$, we are guaranteed the existence of a neighbourhood $\bs U$ of the diagonal $\bs D$ such that $\bs Z|_{(\bs U\setminus \bs D)\times[0,t_0]}>0$. So there must indeed exist an \emph{off-diagonal} pair $(\bs{x},\bs{y}) \in (\bs{M}\times \bs{M})\setminus\bs D$ attaining a zero minimum for\footnote{The same conclusion can be reached by analyzing the \emph{second} derivatives of $\bs Z$ in case $\varphi'(0,t)\equiv 1$, but we will in any case eventually choose $\varphi$ to satisfy the strict inequality $\varphi'(0,t)<1$.} $\bs Z(\cdot,\cdot,t_0)$. 

Proposition \ref{prop:spatial-minimum} now reduces to the following two cases depending on the location of the spatial minimum.

\textbf{Case (a):} $\sign(\bs{x}) = \sign(\bs{y})$, so that $0\leq \bs{Z}(\bs{x},\bs{y},t) = Z(x,y,t) $, with equality at $t_0$, and hence
\begin{align*}
 0 \geq \pr_t Z(x,y,t_0) ={}& -\kappa_x \inner{\frac{\gamma(x)-\gamma(y)}{|\gamma(x)-\gamma(y)|}}{N_x} -\kappa_y \inner{\frac{\gamma(y)-\gamma(x)}{|\gamma(x)-\gamma(y)|}}{N_y}\\
 & - \varphi \pr_t \bs{L} - \varphi' \pr_t\ell + \frac{\ell}{\bs{L}} \varphi' \pr_t \bs{L} - \bs{L} \pr_t \varphi
 \\
 ={}& -\kappa_x \cos \alpha_x + \kappa_y \cos \alpha_y - \varphi \pr_t \bs{L} - \varphi' \pr_t\ell + \frac{\ell}{\bs{L}} \varphi' \pr_t \bs{L} - \bs{L} \pr_t \varphi\,.
\end{align*}

Note that
\[
\pr_t \bs{L} = -2 \int_{\Gamma_t} \kappa^2 ds\, \;\;\text{and}\;\; \pr_t \ell = -\int_{[x:y]} \kappa^2 ds\,,
\]
where $[x:y]$ is the interval between $x$ and $y$. Applying the spatial minimum condition of Proposition \ref{prop:spatial-minimum} now yields \eqref{eq:case-a}.

\textbf{Case (b):} $\sign(\bs{x}) \ne \sign(\bs{y})$, so that $0=\bs{Z}(\bs{x},\bs{y},t_0) = \tilde{Z}(x,y,t_0) = \bar{Z}(x,y,z,t_0)$ for some $z\in S$, and hence 
\begin{align*}
 0 \geq \pr_t \bar{Z}(x,y,z,t_0) ={}& -\kappa_x \inner{\frac{\gamma(x)-z}{|\gamma(x)-z|}}{N_x} -\kappa_{y} \inner{\frac{\gamma(y)-z}{|\gamma(y)-z|}}{N_{y}}  \\
 & - \varphi \pr_t \bs{L} - \varphi' \pr_t \tilde{\ell} + \frac{\tilde{\ell}}{\bs{L}} \varphi' \pr_t \bs{L} - \bs{L} \pr_t \varphi \\
={}& -\kappa_x \cos \beta_{x} - \kappa_{y} \cos \beta_{y} - \varphi \pr_t \bs{L} - \varphi' \pr_t\tilde{\ell} + \frac{\tilde{\ell}}{\bs{L}} \varphi' \pr_t \bs{L} - \bs{L} \pr_t \varphi. 
\end{align*}

Now note that $\pr_t \tilde{\ell} = -\int_{[\bs{x}:\bs{y}]} \kappa^2 ds$, where $[\bs{x}:\bs{y}]$ is the shorter portion of $\bs{M}\setminus \{\bs{x},\bs{y}\}$. Applying the spatial minimum condition of Proposition \ref{prop:spatial-minimum} now yields \eqref{eq:case-b}
%
\end{proof}


\subsection{Lower bounds for the chord-arc profile}

Note that the length is monotone non-increasing under free boundary curve shortening flow, and hence attains a limit as $t\to T$. In order to (crudely) estimate the curvature integrals in Proposition \ref{prop:comparison equation}, we will make use of the following lemma.

\begin{lemma}\label{lem:annoying angle lemma}
Let $\epsilon\in(0,\frac{\pi}{2})$. 
Suppose $\Gamma$ is a curve in $\Omega$ which meets $S=\pd\Omega$ orthogonally at $\pd\Gamma=\{z_0,z_1\}$. Denote by $L$ the length of $\Gamma$, and $C= \sup_{S\cap B_{3L}(\pr\Gamma)} \kappa^S$.

 If $L (1+C) \leq \frac{\epsilon}{100}$, then $|\measuredangle(N^S_{z_0}, N^S_{z_1})| \leq \frac{\epsilon}{2}$. Moreover, if $x,y\in \Gamma$, then any $z\in S$ realising $\tilde{d}(x,y) = d(x,z) + d(y,z)$ must satisfy $|\measuredangle(N^S_z, N^S_{z_i})| \leq \frac{\epsilon}{2}$ for each $i=0,1$. 
\end{lemma}
\begin{proof}
To prove the first claim, we first estimate
\[|\measuredangle(N^S_{z_0}, N^S_{z_1})| = \int_{[z_0:z_1]} \kappa^S ds \leq \ell_S(z_0, z_1) C,\]
where $[z_0:z_1]$ denotes the portion of $S$ between $z_0$ and $z_1$. 
On the other hand, by \cite[Lemma 3.5]{EGF},
\[ \ell_S(z_0,z_1) \leq \frac{2}{C} \sin^{-1}\left(\frac{C}{2} d(z_0,z_1)\right) \leq \frac{2}{C} \sin^{-1}\left(\frac{\epsilon}{200}\right) \leq\frac{\epsilon}{2C}\,.\] 

For the second claim, let $B$ be the ball of radius $3L$ about $z_0$, so that the $2L$-neighbourhood of $\Gamma$ is contained in $B$; in particular, any $z\in S$ realising $\tilde{d}(x,y) = d(x,z) + d(y,z)$ must also lie in $B$. Then as above we have
\[ |\measuredangle(N^S_z, N^S_{z_i})| \leq \ell_S(z,z_i) C\]
and
\[ \ell_S(z,z_i) \leq \frac{2}{C} \sin^{-1}\left(\frac{C}{2} d(z,z_i)\right) \leq \frac{2}{C} \sin^{-1}\left(\frac{6\epsilon}{50}\right) \leq\frac{\epsilon}{2C}\,.\qedhere\] 
\end{proof}

The following theorem provides a uniform lower bound for the chord-arc profile $\bs\psi(\cdot,t)$ of $\Gamma_t$ so long as $L(t)\to 0$ as $t\to T$. Recall that the evolution of any compact curve in a convex domain will always remain in some compact set.

\begin{theorem}
Let $\{\Gamma_t\}_{t\in[0,T)}$ be a compact free boundary curve shortening flow in a convex domain $\Omega$ which remains in the compact set $K$. Suppose that $L(0)(1+C) \leq \frac{\epsilon}{1000}$ for some $\varepsilon\in(0,\frac{1}{10})$, where $C:=\sup_{S\cap K} \kappa^S$, $S:=\pr\Omega$. Given any $c\in(0,\frac{1}{100})$, if the inequality
\[
\bs{\psi}(\delta,t)  \geq  c\bs L(t)\left(\sin\left((\pi-\epsilon)\frac{\delta}{\bs L(t)}+ \frac{\epsilon}{2}\right)-\sin\left(\frac{\epsilon}{2}\right)\right)
\]
holds at $t=0$, then it holds for all $t\in [0,T)$.
\end{theorem}
\begin{proof}
Set $\varepsilon_0:=\frac{1}{10}\varepsilon$ and $c_0:=\frac{1}{100}$. Observe that the function $\varphi\in C^\infty([0,1])$ defined by
\begin{equation}\label{eq:barrier defn}
\varphi(\zeta)\doteqdot c\left(\sin\left((\pi-\varepsilon)\zeta+ \frac{\varepsilon}{2}\right)-\sin\left(\frac{\varepsilon}{2}\right)\right)
\end{equation}

is symmetric about $\zeta=\frac{1}{2}$ and strictly concave, and has subunital gradient. So it is an admissible comparison function. Forming the auxiliary function $\bs{Z}$ as in \eqref{eq:bs Z} with this choice of $\varphi$, we have $\bs{Z}\geq0$ at $t=0$ by supposition, and we will show that $\bs{Z}\geq 0$ for all $t>0$. Indeed, if, to the contrary, $t_0 \doteqdot  \sup\{t\in[0,T):\bs{Z}(\cdot,\cdot,t)\ge 0\}<T$, then 
we shall arrive at an absurdity via Proposition \ref{prop:comparison equation}. 

Let $\bs{x},\bs{y}$ be as given by Proposition \ref{prop:comparison equation}, and define
\[
\Theta \doteqdot  \frac{1}{2} \left( \bs{L} \int_{\Gamma_t} \kappa^2 ds \right)^\frac{1}{2}\;\;\text{and}\;\;\bs{\omega} \doteqdot  \frac{1}{2} \left( \bs{\ell} \int_{[\bs{x}:\bs{y}]} \kappa^2 ds \right)^\frac{1}{2}\,.
\]
Applying the Cauchy--Schwarz inequality yields
\[
2\Theta \geq \sqrt{2} \int_{\Gamma_t} |\kappa|ds.
\]
On the other hand, by the free boundary condition and the theorem of turning tangents,
\[
\cos\left(\int_{\Gamma_t} \kappa \,ds\right) = \cos \measuredangle(T_{z_1}, T_{z_0})= \cos \measuredangle(N^S_{z_1} , -N^S_{z_0}) = \cos\left( \pi +\measuredangle(N^S_{z_1},N^S_{z_0})\right),
\]
where $z_i$ are the endpoints of $\Gamma_t$. By Lemma \ref{lem:annoying angle lemma}, we may estimate $|\measuredangle(N^S_{z_1}, N^S_{z_0})| <\epsilon_0$, and hence \[(2\Theta)^2 \geq 2(\pi-\epsilon_0)^2.\]

We may similarly estimate
\[
2\bs{\omega} \geq  \int_{[\bs{x}:\bs{y}]} |\kappa|ds\,.
\]
If $\sign(\bs{x})=\sign(\bs{y})$, then $\int_{[\bs{x}:\bs{y}]} \kappa\,ds$ measures the turning angle from $T_x$ to $T_y$; in particular
\[
\cos\left(\frac{1}{2}\int_{[\bs{x}:\bs{y}]} \kappa\,ds\right) = \cos \frac{\measuredangle(T_x,T_y)}{2} = \cos\left(\frac{\alpha_x - \alpha_y}{2}\right) = \cos\left(\alpha_x -\frac{\pi}{2}\right) =  \varphi'.
\] Here we have used that $\alpha_x+\alpha_y=\pi$ and $\varphi' = -\sin\alpha_x$. Thus in case (a) we may estimate
\[ (2\bs{\omega})^2 \geq 2 (\cos^{-1}\varphi')^2.\] 
	
If $\sign(\bs{x})\neq \sign(\bs{y})$, then again by reversing the parametrisation if necessary, we may assume without loss of generality that $[\bs{x}:\bs{y}]$ is precisely the interval $[y,x]$. In particular, $\int_{[\bs{x}:\bs{y}]} \kappa$ gives the turning angle from $T_x$ to $T_{z_0}$ plus the turning angle from $T_{y}$ to $T_{z_0}$, where $z_0 = \gamma(0)$. But now
\[
\begin{split}
\measuredangle(T_x, T_{z_0}) &= \measuredangle(T_x, T^S_{z_0})- \frac{\pi}{2}\\
& = \measuredangle(T_x, T^S_z) +\measuredangle(T^S_z, T^S_{z_0}) - \frac{\pi}{2} \\&= \measuredangle(T_x, T^S_z) +\measuredangle(N^S_z, N^S_{z_0}) - \frac{\pi}{2} \\
& = \beta_{x} - \theta_{x} + \measuredangle(N^S_z, N^S_{z_0}) - \frac{\pi}{2} . 
\end{split}
\]

Arguing similarly for $T_{y}$, we estimate
\[\begin{split}
 \cos\left(\frac{1}{2}\int_{[\bs{x}:\bs{y}]} |\kappa|\,ds\right) &\geq \cos\left( \frac{1}{2} \int_{[0,x]} \kappa\,ds + \frac{1}{2}\int_{[0,y]} \kappa\,ds \right)
\\& =  \cos\left(\frac{1}{2}(\measuredangle(T_x, T_{z_0}) + \measuredangle(T_{y}, T_{z_0}) )\right)
\\& = \cos\left(\frac{1}{2}(\beta_{x} - \theta -\frac{\pi}{2} + \beta_{y} + \theta -\frac{\pi}{2}) + \measuredangle(N^S_z, N^S_{z_0})\right)
\\& = \cos\left(\frac{1}{2}(\beta_{x} + \beta_{y}-\pi)+ \measuredangle(N^S_z, N^S_{z_0})\right) 
\\&= \cos\left(\frac{\pi}{2} -\beta -  \measuredangle(N^S_z, N^S_{z_0})\right)\,.
\end{split}\]
Here we have used $\beta_{x} = \beta_{y}=\beta$. By Lemma \ref{lem:annoying angle lemma}, we may estimate $|\measuredangle(N^S_z, N^S_{z_0})| \leq \frac{\epsilon_0}{2}$. Recall also that $\sin\beta=\varphi'$; since $\varphi$ satisfies $|\varphi'| \leq c(\pi-\varepsilon)$ and $c<\frac{1}{\pi}$, we have
\[
\vert\sin\beta\vert=\vert\varphi'\vert\le c(\pi-\varepsilon)=c(\pi-10\varepsilon_0)\le 1-\tfrac{10}{\pi}\varepsilon_0\le 1-\tfrac{2}{\pi}\varepsilon_0\le \sin(\tfrac{\pi}{2}-\varepsilon_0)\,,
\]
and hence $|\beta| < \frac{\pi}{2}-\epsilon_0$.

 In particular, $\cos^{-1}\varphi' = \frac{\pi}{2}-\beta>\epsilon_0$. It follows that in case (b) we may estimate
\[(2\bs{\omega})^2 \geq 2\left(\cos^{-1}\varphi' -\frac{\epsilon_0}{2}\right)^2.\] 

Since by Lemma \ref{lem:annoying angle lemma} we also have $|\measuredangle(N^S_{z_0}, N^S_{z_1})| \leq \frac{\epsilon_0}{2}$, we have shown that, in either case,
\begin{equation}\label{eq:comparison-final}
\bs{L}^2 \pr_t\varphi \geq 4\varphi''  + 4 \left( \varphi - \varphi' \frac{\bs\ell}{\bs{L}}\right) \left(\pi- \epsilon_0\right)^2 +4 \frac{\bs{L}}{\bs{\ell}}  \varphi'\left(\cos^{-1}\varphi'-\frac{\epsilon_0}{2}\right)^2\,.
\end{equation}

We claim that this is impossible when $\varphi$ is given by \eqref{eq:barrier defn}, proving the theorem. To see this, we first estimate
\ba\label{eq:remainder of bad term}
\varphi''+(\pi-\varepsilon_0)^2\varphi={}&c\left[\left((\pi-\varepsilon_0)^2-(\pi-\varepsilon)^2\right)\sin([\pi-\varepsilon]z+\tfrac{\varepsilon}{2})-(\pi-\varepsilon_0)^2\sin\tfrac{\varepsilon}{2}\right]\nonumber\\
\ge{}&
c\Big[(\varepsilon-\varepsilon_0)\big(2\pi-\epsilon_0-\epsilon\big)\left(1-(\pi-\varepsilon)^2(\tfrac{1}{2}-z)^2\right)-(\pi-\varepsilon_0)^2\sin\tfrac{\varepsilon}{2}\Big]\nonumber\\
>{}&-c(\varepsilon-\varepsilon_0)\big(2\pi-\epsilon_0-\epsilon\big)(\pi-\varepsilon)^2(\tfrac{1}{2}-z)^2,
\ea
where we have used the quadratic estimate $\sin X \geq 1- \frac{1}{2}\left(\frac{\pi}{2}-X\right)^2$ and the smallness of $\varepsilon_0=\frac{1}{10}\varepsilon$ to ensure that $(\epsilon-\epsilon_0)(2\pi-\epsilon_0-\epsilon)> (\pi-\epsilon_0)^2\sin \tfrac{\epsilon}{2}$. 




Next, note that (as $c(\pi-\epsilon)<1$) the function $f(X):=\cos^{-1}(c(\pi-\varepsilon)\cos X)$ is convex on $[0,\frac{\pi}{2}]$; 
 in particular,
\[
f(X) \geq f(\tfrac{\pi}{2})+(X-\tfrac{\pi}{2})f'(\tfrac{\pi}{2})=\tfrac{\pi}{2} - c(\pi-\epsilon)(\tfrac{\pi}{2}-X)\,.
\]
As $\varphi'=c(\pi-\varepsilon)\cos((\pi-\varepsilon)\zeta+\frac{\varepsilon}{2})$, this implies that 
\[
\cos^{-1}\varphi'-\tfrac{\varepsilon_0}{2}\ge \tfrac{\pi-\varepsilon_0}{2}-c(\pi-\varepsilon)^2(\tfrac{1}{2}-z) \geq 0. 
\]
Thus, expanding $z^2=(\frac{1}{2}-(\frac{1}{2}-z))^2$, we find that
\[\begin{split}
 \left(\cos^{-1}\varphi'-\tfrac{\varepsilon_0}{2}\right)^2-(\pi-\epsilon_0)^2 z^2 &\geq  \left((\pi-\epsilon_0)^2-c(\pi-\epsilon_0)(\pi-\epsilon)^2\right)(\tfrac{1}{2}-z)  \\&\qquad+ \left(c^2(\pi-\epsilon)^4 -(\pi-\epsilon_0)^2\right)(\tfrac{1}{2}-z)^2
 \\&\geq (\pi-\epsilon_0)(\tfrac{\pi-\epsilon_0}{2} - c(\pi-\epsilon)^2)(\tfrac{1}{2}-z)\,.
\end{split}
\]
Applying the linear estimate $\cos(X)\ge \tfrac{2}{\pi}(\tfrac{\pi}{2}-X)$ now gives 
\begin{equation}\label{eq:good term}
\frac{\varphi'}{z}\left(\left(\cos^{-1}\varphi'-\tfrac{\varepsilon_0}{2}\right)^2-(\pi-\epsilon_0)^2 z^2 \right) \geq \frac{2c}{\pi z}(\pi-\epsilon_0)(\tfrac{\pi-\epsilon_0}{2} - c(\pi-\epsilon)^2)(\pi-\epsilon)^2(\tfrac{1}{2}-z)^2.
\end{equation}
In particular, since $\frac{2}{z} \geq 4$, $c_0(\pi-\epsilon)^2 < c_0\pi^2<\tfrac{\pi}{4}$, and $\epsilon<\frac{1}{10}$, taking the sum of (\ref{eq:remainder of bad term}) and (\ref{eq:good term}) gives 
\[ \phi'' + (\pi-\epsilon_0)^2\phi +\frac{\varphi'}{z}\left(  \left(\cos^{-1}\varphi'-\tfrac{\varepsilon_0}{2}\right)^2-(\pi-\epsilon_0)^2 z^2 \right) >0,\]
which contradicts (\ref{eq:comparison-final}) as claimed. This completes the proof.
%

\end{proof}

Note that if $L(t)\to 0$ as $t\to T$, then the condition $\left(1+\sup_{S\cap K} \kappa^S\right) L(0)\leq \frac{\epsilon}{1000}$ is eventually satisfied for any $\varepsilon\in(0,\frac{1}{10})$. 
So, unless $L(t)\not\to 0$ as $t\to T$, we can always find some $c>0$ and $t_0\in [0,T)$ such that the theorem applies to $\{\Gamma_t\}_{t\in[t_0,T)}$.

On the other hand, if $L(t)\not\to0$, then we may apply the following (very) crude bound.

\begin{theorem}
Let $\{\Gamma_t\}_{t\in[0,T)}$ be a free boundary curve shortening flow with $\bs{\psi}(\delta,0) \geq cL(0)\sin\left(\frac{\pi\delta}{L(0)}\right)$ and $T<\infty$. If $\bs{L}_T \doteqdot \lim_{t\to T}\bs{L}(t)>0$, then 
\[ \bs{\psi}(\delta,t) \geq  c\bs{L}(t)e^{-\frac{4\pi^2T}{\bs{L}_T^2}} \sin\left(\frac{\pi\delta}{\bs{L}(t)}\right) 
\]
for all $t\in[0,T)$. 
\end{theorem}
\begin{proof}
We introduce the modified time coordinate $\tau \doteqdot  \int_0^t \frac{1}{\bs{L}^2} dt$ and take $\varphi(\zeta,\tau) \doteqdot c e^{-4\pi^2 \tau} \sin(\pi \zeta)$. Then $\bs{Z}\geq0$ at $t=0$ by supposition, and we will show that $\bs{Z}\geq 0$ for all $t>0$. As above, suppose to the contrary that $\bs{Z} <0$ at some positive time, so that $t_0 \doteqdot  \sup \{t: \bs{Z}(\cdot,\cdot,t)\ge 0\}\in(0,T)$ and we can find $(\bs{x},\bs{y})\in (\bs M\times \bs M)\setminus\bs D$ such that $\bs{Z}\geq 0$ for $t\in [0,t_0]$ and $Z(\bs{x},\bs{y},t_0)=\min \bs{Z}(\cdot,\cdot,t_0)=0$. 

Noting that $\pr_\tau \varphi = \bs{L}^2 \pr_t \varphi$, and that (unless $\pr\Omega$ is flat) all the terms in Proposition \ref{prop:comparison equation} involving spatial derivatives of $\varphi$ are strictly positive (except $\varphi''$), we find that
\[
0 < 4\varphi'' - \pr_\tau\varphi=0\,,
\]
which is absurd.  We conclude that that $\bs{Z}\geq 0$ for all $t\in [0,T)$. The claim follows since $\tau \leq \frac{1}{\bs{L}_T^2}T$. 
\end{proof}


\subsection{Boundary avoidance}

The chord-arc bound immediately yields the following ``quantitative boundary avoidance'' estimate.

Given a curve $\gamma: M\to \Omega$, we shall denote by $\lambda:M\to \mathbb{R}$ the distance to the nearest endpoint; if $\gamma$ is parametrised by arclength, then $\lambda(x) = \min\{x, L-x\}$. 

\begin{proposition}
\label{prop:boundary-avoidance}
Let $\{\Gamma_t\}_{t\in[0,T)}$ be a compact free boundary curve shortening flow in a convex domain $\Omega$. Given any $\delta>0$, there exists $\varepsilon= \varepsilon(\Gamma_0, \Omega,\delta)>0$ such that
\[
\lambda(x,t)>\delta\;\;\implies\;\; d(\gamma(x,t),\pr\Omega)>\epsilon\,.
\]
\end{proposition}
\begin{proof}
The chord-arc estimate yields $\bs{d}/\bs{\ell} >c>0$; in particular,
\[
d(\gamma(x,t),\pd\Omega)=\frac{1}{2}\bs d(x,-x,t)\ge \frac{c}{2}\bs\ell(x,-x,t)=c\lambda(x,t)\,.\qedhere
\]
\end{proof}

\section{Convergence to a critical chord or a round half-point}
\label{sec:grayson}

We now exploit the chord-arc estimate to rule out collapsing at a finite time singularity, resulting in a free-boundary version of Grayson's theorem, Theorem \ref{thm:fbGrayson}. Given $z\in S=\pr\Omega$, we denote by $T_z\Omega$ the halfplane $\{p\in\R^2: \langle p, N^S(z) \rangle \le 0\}$. 

\begin{theorem}\label{thm:fbGrayson}
Let $\Omega \subset \mathbb{R}^2$ be a convex domain of class $C^2$ and let $\{\Gamma_t\}_{t\in [0,T)}$ be a maximal free boundary curve shortening flow starting from a properly embedded interval $\Gamma_0$ in $\Omega$. Either:
\begin{enumerate}
\item[(a)] $T=\infty$, in which case $\Gamma_t$ converges smoothly as $t\to \infty$ to a chord in $\Omega$ which meets $\pd\Omega$ orthogonally; or 
\item[(b)] $T<\infty$, in which case $\Gamma_t$ converges uniformly to some $z\in \pr\Omega$, and
\[
\tilde\Gamma_t\doteqdot \frac{\Gamma_t -z}{\sqrt{2(T-t)}}
\]
converges uniformly in the smooth topology as $t\to T$ to the unit semicircle in $T_z\Omega$.
\end{enumerate}
\end{theorem}

\subsection{Long time behaviour} We first address the long-time behaviour. 

\begin{proof}[Proof of Theorem \ref{thm:fbGrayson} part (a)]

First recall that, since $\Omega$ is convex, $\Gamma_t$ remains in some compact subset $K\subset \Omega$ for all time.

Next, observe that the length approaches a positive limit as $t\to\infty$. Indeed, a limit exists due to the monotonicity
\ba
\frac{dL}{dt}={}&-\int\kappa^2\,ds\label{eq:evolve length}\\
\le{}&0\,,\nonumber
\ea
and the limit cannot be zero: If it were zero, then we could eventually enclose $\Gamma_t$ by a small convex arc which meets $\pd\Omega$ orthogonally. The latter would contract to a point on the boundary in finite time (in accordance with Stahl's theorem), whence the avoidance principle would force $\Gamma_t$ to become singular in finite time, contradicting $T=\infty$. 

Then integrating \eqref{eq:evolve length} from time $0$ to $\infty$, we find that
, for every $\varepsilon>0$, we can find $t_\varepsilon<\infty$ such that
\begin{equation}\label{eq:aeL2est}
\int_{\Gamma_t}\kappa^2\,ds<\varepsilon
\end{equation}
for \emph{almost every} $t\ge t_\varepsilon$. We can bootstrap this to full convergence as follows (cf. \cite{MR1046497,MR979601}): integrating by parts and applying the boundary condition yields
\bann
\frac{d}{dt}\int_{\Gamma_t}\kappa^2\,ds={}&\int_{\Gamma_t}\left(2\kappa(\Delta\kappa+\kappa^3)-\kappa^4\right)\\
={}&2\sum_{\pd\Gamma_t}\kappa^S\kappa^2-2\int_{\Gamma_t}\vert\cd\kappa\vert^2\,ds+\int_{\Gamma_t}\kappa^{4}\\
\le{}&2C\sum_{\pd\Gamma_t}\kappa^2-2\int_{\Gamma_t}\vert\cd\kappa\vert^2+\max_{\Gamma_t}\kappa^2\int_{\Gamma_t}\kappa^2\,,
\eann
where $C\doteqdot \max_{S\cap K}\kappa^S$, $S=\pr\Omega$. Since (by Stahl's theorem, say) $\min_{\Gamma_t}\vert\kappa\vert=0$ for each $t$, the fundamental theorem of calculus and the H\"older inequality yield
\bann
\max_{\Gamma_t}\kappa^2\le{}&\left(\int_{\Gamma_t}\vert\cd\kappa\vert\right)^2\le L\int_{\Gamma_t}\vert\cd\kappa\vert^2\le L_0\int_{\Gamma_t}\vert\cd\kappa\vert^2\,,
\eann
while the fundamental theorem of calculus and the Cauchy--Schwarz inequality yield
\bann
2C\sum_{\pd\Gamma_t}\kappa^2\le2C\int_{\Gamma_t}\vert\cd\kappa^2\vert
\le{}&4C^2 \int_{\Gamma_t}\kappa^2+\int_{\Gamma_t}\vert\cd\kappa\vert^2\,.
\eann
Thus,
\ba\label{eq:L2 subexp growth}
\frac{d}{dt}\int_{\Gamma_t}\kappa^2\,ds\le{}&4C^2 \int_{\Gamma_t}\kappa^2+\left(L_0\int_{\Gamma_t}\kappa^2-1\right)\int_{\Gamma_t}\vert\cd\kappa\vert^2\,.
\ea
Now, given any $\varepsilon\in (0,\frac{L_0}{2})$ we can find $t_\varepsilon$ such that \eqref{eq:aeL2est} holds for almost every $t\ge t_\varepsilon$. But then by \eqref{eq:L2 subexp growth} there is a a dense set of times $t'\ge t_\varepsilon$ such that $\int_{\Gamma_{t}}\kappa^2\,ds\le 2\varepsilon$ for \emph{every} $t\in [t', t'+\delta]$, where $\delta\doteqdot\frac{\log2}{4C^2}>0$. It follows that \[\int_{\Gamma_t} \kappa^2 \,ds \to 0\] as claimed. 

In particular, with respect to an arclength parametrization, the $W^{2,2}$ norm of $\gamma(\cdot,t):[0,L(t)]\to\Omega$ is bounded independent of $t$. Reparametrizing by a family of uniformly controlled diffeomorphisms $\phi(\cdot,t):[0,L(t)]\to [0,1]$, we obtain a family of embeddings $\tilde\gamma(\cdot,t)\doteqdot \gamma(\phi^{-1}(\cdot,t),t)\in W^{2,2}([0,1];\R^2)$ with uniformly bounded $W^{2,2}$-norm and $\vert\tilde\gamma'\vert$ uniformly bounded from below. Since the Sobolev embedding theorem then implies uniform bounds in $C^{1,\alpha}([0,1];\R^2)$ for every $\alpha<\frac{1}{2}$, the Arzel\`a--Ascoli theorem yields, for any sequence of times $t_j\to\infty$, a subsequence along which $\tilde\gamma(\cdot,t_j)$ converges in $C^{1,\alpha}([0,1];\R^2)$, for every $\alpha<\frac{1}{2}$, to a limit immersion $\gamma_{\infty}\in W^{2,2}([0,1];\R^2)$ satisfying $\kappa=0$ in the weak sense and orthogonal boundary condition. In particular, $\gamma_\infty$ must parametrise a straight line segment which meets $\pr \Omega$ orthogonally; we call such a segment a \textit{critical chord}.

We need to show that the limit chord, which we denote by $\sigma$, is unique. To achieve this, we will show that the \emph{endpoints} of $\Gamma_t$ converge to those of $\sigma$.

\begin{claim}\label{claim:finite crossing}
The endpoints of $\Gamma_t$ cross those of $\sigma$ at most finitely many times.
\end{claim}
\begin{proof}[Proof of Claim \ref{claim:finite crossing}]
Observe that the height function $y\doteqdot \langle \gamma ,  N^\sigma\rangle$, where $N^\sigma$ is a choice of unit normal to $\sigma$, satisfies
\begin{equation}
\label{eq:height}
\left\{\begin{aligned}(\pd_t-\Delta)y={}&0\;\;\text{in}\;\;\Gamma_t\setminus\pd\Gamma_t\\
 \langle \nabla y, N^S\rangle ={}&\langle N^\sigma, N^S\rangle  \;\;\text{on}\;\;\pd\Gamma_t\,.
\end{aligned}\right.
\end{equation}
In particular, the conormal derivative $\langle \nabla y , N^S\rangle$ vanishes at any boundary zero of $y$. We claim that, unless $\{\Gamma_t\}_{t\in[0,T)}$ is the stationary chord $\Gamma_t\equiv \sigma$, the boundary of $\Omega$ is a \emph{strict zero sink} for $y$; that is, 
\begin{itemize}
\item if $p\in\pd\Gamma_{t_0}$ is a zero of $y$ at a positive time $t_0$, then we can find $r>0$ such that $\Gamma_t\cap B_r$ contains a zero of $y$ for all $t\in(t_0-r^2,t_0)$ but not for $t\in(t_0,t_0+r^2)$.
\end{itemize}
Indeed, if $p\in\pd\Gamma_{t_0}$ is a zero of $y$ for $t_0>0$, then $\Gamma_{t_0}$ lies locally (and nontrivially) to one side of $C$ in a neighbourhood $B$ of $p$ (above, say). But then, since $\cd y\cdot N^S|_p=0$, the strong maximum principle implies that $y>0$ in $B$ for a short time. On the other hand, if we can find $r>0$ such that $\Gamma_t\cap B_r(p)$ does not contain a zero of $y$ for $t\in (t_0-r^2,t_0)$, then the Hopf boundary point lemma implies that $\cd y\cdot N^S<0$ at $p$, which contradicts (\ref{eq:height}). Since, with respect to a parametrization $\tilde{\gamma}:[0,1]\times[0,T)\to \Omega$ for $\{\Gamma_t\}_{t\in[0,T)}$ over a fixed interval, $y$ satisfies a linear diffusion equation with suitably bounded coefficients, it now follows from Angenent's Sturmian theory \cite{Ang88} that the zero set of $y$ is finite and non-increasing at positive times (and in fact strictly decreasing each time $y$ admits a degenerate or boundary zero). The claim follows. 
\end{proof}

\begin{claim}\label{claim:finite direction change}
The endpoints of $\Gamma_t$ change direction at most finitely many times.
\end{claim}
\begin{proof}[Proof of Claim \ref{claim:finite direction change}]
Recall that the curvature satisfies \cite{Stahl96a}
\begin{equation}\label{eq:evolvekappa}
\left\{\begin{aligned}(\pd_t-\Delta)\kappa={}&\kappa^3\;\;\text{in}\;\;\Gamma_t\\
\langle \nabla \kappa, N^S\rangle={}&\kappa^S\kappa\;\;\text{on}\;\;\pd\Gamma_t\,.
\end{aligned}\right.
\end{equation}
In particular, the conormal derivative $\langle \nabla \kappa , N^S\rangle$ vanishes at any boundary zero of $\kappa$. Thus, applying essentially the same argument as in Claim \ref{claim:finite crossing}, we find that the number of zeroes of $\kappa$ is finite and non-increasing at positive times, and strictly decreasing any time $\kappa$ admits a degenerate or boundary zero. The claim follows.
\end{proof}

Claims \ref{claim:finite crossing} and \ref{claim:finite direction change} imply that the endpoints of $\Gamma_t$ converge to some pair of limit boundary points, which must then be the endpoints of $\sigma$, and we may thus conclude that the limit chord is indeed unique. This proves convergence of $\tilde\gamma(\cdot,t)$ to a chord in $C^{1,\alpha}([0,1];\R^2)$ as $t\to \infty$. In particular, we may eventually write $\Gamma_t$ as a graph over the limit chord with uniformly H\"older controlled height and gradient, so the Schauder estimate \cite[Theorem 4.23]{Lieberman} 
and interpolation yield convergence in the smooth topology.
\end{proof}

We present two routes to case (b) of Theorem \ref{thm:fbGrayson}: one using (smooth) intrinsic blowups in the spirit of Hamilton \cite{HamiltonPinched,MR1369140} and Huisken \cite{Huisken96}; and one using (weak) extrinsic blowups in the spirit of White \cite{Wh03} (cf. Schulze \cite{Schulze}). 
We present the extrinsic method first, as (utilising the powerful theory of free-boundary Brakke flows developed by Edelen \cite{Edelen}) it quickly reduces the problem to ruling out multiplicity of blowup limits, which the chord-arc bound easily achieves. The intrinsic method is more elementary but requires the adaptation of a number of (interesting) results to the free boundary setting (for instance a monotonicity formula for the total curvature).

\subsection{Extrinsic blowup}

We follow the treatment of Schulze \cite{Schulze} (taking care to explain the modifications required in order to contend with the boundary condition). We begin by classifying the tangent flows following Edelen's theory of free boundary Brakke flows \cite{Edelen} (cf. \cite[Theorem 6.4]{Edelen} and \cite[Theorem 6.9]{Buckland}).


We say that a spacetime point $(x_0,t_0)$ is reached by a free boundary curve shortening flow if (for instance) its Gaussian density is positive: $\Theta(x_0,t_0)>0$. 

\begin{lemma}
\label{lem:tangent-flow}
Let $(x_0,t_0)$ be a point in spacetime reached by the free boundary curve shortening flow $\{\Gamma_t\}_{t\in I}$. For any sequence of scales $\lambda_i\to \infty$, there is a subsequence such that 
\[\{\Gamma_t^{i}\doteqdot\lambda_i(\Gamma_{t_0+\lambda_i^{-2}t}-x_0))\}_{t\in [-\lambda_i^2t_0,0)} \to \{\Gamma^\infty_t\}_{(-\infty,0)}\] graphically and locally smoothly, with multiplicity 1, where $\{\Gamma_t^\infty\}_{t\in(-\infty,0)}$ is one of the following: 
\begin{enumerate}[(a)]
\item a static line through the origin; 
\item a static half-line from the origin;
\item a shrinking semicircle.
\end{enumerate}
\end{lemma}

\begin{proof}
By \cite[Theorems 4.10 and 6.4]{Edelen}, there is certainly a subsequence along which the flows $\{\Gamma_t^i\}$ converge as free boundary Brakke flows to some self-shrinking free boundary Brakke flow $(\mu_t)$. 

Moreover, a slight modification of the proof of Edelen's reflected, truncated monotonicity formula \cite[Theorem 5.1]{Edelen} reveals that
\begin{equation}\label{eq:RHS of MF}
\int_{\tilde\Gamma{}^i_{t}}\left\vert\frac{\tilde x^\perp}{-2t}+\tilde\kappa\tilde\nu\right\vert^2\tilde\phi\tilde\rho+\int_{\Gamma^i_t}\left\vert\frac{x^\perp}{-2t}+\kappa\nu\right\vert^2\phi\rho\to 0\,,
\end{equation}
in $L^1_{\mathrm{loc}}((-\infty,0))$. (Indeed, the second term in the penultimate line of the estimate at the bottom of page 115 of \cite{Edelen} is discarded by Edelen, and one need not discard all of the first term on that line in producing the estimate on the top of page 116.) This gives a corresponding (extrinsic) $L^2_{\mathrm{loc}}$ bound for $\kappa$: for almost every $t\in(-\infty,0)$ and for every $R>0$,
\[
\int_{\Gamma^i_t\cap B_R}\kappa^2\le C
\]
independent of $i$. 

Consider such a time $t=\tau$. Since $(x_0,t_0)$ is reached by the flow, there must exist points $p_i \in \Gamma^i_\tau$ converging to some limit  $p$. We may consider an arclength parametrisation $\gamma^i_\tau$ of each $\Gamma^i_\tau$ such that $\gamma^i_\tau(0)=p_i$. By applying the Sobolev embedding theorem (cf. the proof of case (a) of Theorem \ref{thm:fbGrayson} above), there will be a further subsequence along which the maps $\gamma^i_\tau$ converge, in the $C^{1,\alpha}_{\mathrm{loc}}$ topology, to a proper and connected limiting immersion $\gamma_\tau^\infty:M_\infty\to \Pi$ of class $W^{2,2}_{\mathrm{loc}}$. By \eqref{eq:RHS of MF}, this limit satisfies
\[
\vec\kappa=\frac{x^\perp}{-2\tau}
\]
in the weak sense, where $\Pi =T_{x_0}\Omega$ is the whole plane if $x_0\in\mathring{\Omega}$ or the closed halfspace $\{p\in\R^2:\langle p, N^S(x_0)\rangle\le 0\}$ if $x_0\in S=\pd\Omega$. By the $C_{\mathrm{loc}}^{1,\alpha}$ convergence, the limit $\Gamma^\infty_\tau\doteqdot\gamma_\tau^\infty(M_\infty)$ meets $\pd\Pi_{x_0}$ orthogonally at any boundary points, so the Schauder estimates \cite[Theorems 6.2 and 6.30]{GilbargTrudinger} imply that the limit immersion $\gamma^\infty_\tau$ is smooth. Moreover, the boundary avoidance estimate implies that $\mathring{\Gamma}^\infty_\tau\subset\mathring{\Pi}$ (which in particular rules out the barrier $\pr \Pi$ as a limit). 

Now, the only smoothly embedded curves in $\mathbb{R}^2$ or $\mathbb{R}^2_+$ which satisfy the free boundary self-shrinker equation are the (half-)lines through the origin and the (semi)circle of radius $\sqrt{2}$. (Indeed, for $\mathbb{R}^2_+$ a standard reflection argument, as in Proposition \ref{prop:ancient} below, reduces the classification to the planar case \cite{MR2931330}.) 

If $\Gamma^\tau_\infty$ is compact, then the embeddings $\gamma^i_\tau$ converge globally in the $C^{1,\alpha}$ topology. 
In particular, $\Gamma^\infty_\tau$ is topologically a compact interval. The only possibility is that $x_0\in\pd\Omega$, so that $\Pi$ is a half-plane, and $\Gamma^\infty_\tau$ is the semicircle in $\Pi$ centred at the origin with radius $\sqrt{-2\tau}$. 
Thus $(\mu_t)$ is the corresponding shrinking semicircle of multiplicity 1.

If $\Gamma^\infty_\tau$ is noncompact, then it must be a (half-)line. The $C^{1,\alpha}_{\mathrm{loc}}$ convergence means that for any $R>0$, the restriction $\gamma^i_\tau |_{\{|x|<R\}}$ converges to a (half-)line. Using our chord-arc estimate, it follows that there is some $c>0$ (independent of $R$) such that, for large $i$, we have \[\Gamma^i_\tau \cap B_{cR} = \mathrm{im}\left(\gamma^i_\tau |_{\{|x|<R\}}\right) \cap B_{cR}.\] In particular, this implies that $\Gamma^i_\tau$ converges, locally and graphically (in the $C^{1,\alpha}$ topology), to a (half-)line with multiplicity 1. In particular, $(\mu_t)$ must be a stationary (half-)line of multiplicity 1 (which is orthogonal to $\pr\Pi$).

In either case, local regularity for free boundary Brakke flows \cite[Theorem 8.1]{Edelen} implies that the flows $\{\Gamma^i_t\}$ converge locally and graphically, in the smooth topology, to $(\mu_t)$.
\end{proof}

\begin{proof}[Proof of Theorem \ref{thm:fbGrayson} part (b) (extrinsic method)]
By Edelen's local regularity theorem \cite[Theorem 8.1]{Edelen}, if any tangent flow to $\{\Gamma_t\}_{t\in [0,T)}$ at $(x_0,T)$ is a multiplicity one (half-)line, then $(x_0,T)$ is a smooth point of the flow. Since the flow becomes singular in finite time $T<\infty$, by Lemma \ref{lem:tangent-flow} there must be a point $x_0 \in \pr \Omega$ such that every tangent flow at $(x_0,T)$ converges smoothly, with multiplicity 1, to the shrinking semicircle in $\Pi_{x_0}$. 
The result follows (fix any time and consider the scale factors $\lambda = \frac{1}{\sqrt{T-t}}$).
\end{proof}

\subsection{Intrinsic blowup}

We now follow the (smooth) ``type-I vs type-II'' blow-up argument of Huisken \cite{Huisken96}. 

We first exploit the classification of convex ancient planar curve shortening flows to classify smooth free boundary blow-ups. (Recall that a curve in a convex subset $\Omega$ of the plane is \emph{convex} if it is the relative boundary of a convex subset of $\Omega$.)

\begin{proposition}\label{prop:ancient}
The only convex ancient free boundary curve shortening flows in the halfplane $\R^2_+$ are the shrinking round semicircles, the stationary (half-)lines and pairs of parallel (half-)lines, the (half-)Grim Reapers, and the (half-)Angenent ovals.\footnote{Note that there are two geometrically distinct half-Angenent ovals.}
\end{proposition}
\begin{proof}
Let $\{\Gamma_t\}_{t\in[-\infty,\omega)}$ be a convex ancient free boundary curve shortening flow in $\R^2_+$ with nontrivial boundary on $\pd\R^2_+$. By differentiating the 
evolution and boundary value equations \eqref{eq:evolvekappa} for $\kappa$, we find (by induction) that all of the odd-order derivatives of $\kappa$ vanish at $\pd\R^2_+$. We therefore obtain, upon doubling $\Gamma_t$ through (even) reflection across $\pd\R^2_+$, a convex ancient (boundaryless) curve shortening flow in the plane. So the claim follows from the classification from \cite{BLTcsf,DHScsf}.
\end{proof}

In order to ensure \emph{convex} blow-up limits, we will adapt a monotonicity formula of Altschuler \cite[Theorem 5.14]{Altschuler} to the free boundary setting. In order to achieve this, we first need to control the vertices of $\Gamma_t$ under the flow. 
\begin{lemma}\label{lem:Sturm}
Let $\{\Gamma_t\}_{t\in[0,T)}$ be a compact free boundary curve shortening flow in a convex domain $\Omega\subset\R^2$. Unless $\{\Gamma_t\}_{t\in[0,T)}$ is a stationary chord or a shrinking semicircle, the inflection points $\{p\in\Gamma_t:\kappa(p)=0\}$ and interior vertices $\{p\in\mathring{\Gamma_t}:\kappa_s(p)=0\}$ are finite in number 
for all $t>0$. The number of inflection points is non-increasing, and strictly decreases each time $\Gamma_t$ admits a degenerate or boundary inflection point.
\end{lemma}
\begin{proof}
Recall from the proof of part (a) of Theorem \ref{thm:fbGrayson} that the number of zeroes of $\kappa$ is finite and non-increasing at positive times, and strictly decreasing any time $\kappa$ admits a degenerate or boundary zero. In particular, $\kappa_s$ changes sign at most a finite number of times at any boundary point (but could still vanish on an open set of times if $\pd\Omega$ contains flat portions.) Since, with respect to a fixed parametrization $\gamma:[-1,1]\times[0,T)\to\Omega$ for $\{\Gamma_t\}_{t\in[0,T)}$, $\kappa_x$ satisfies a linear diffusion equation, 
we may apply \cite[Theorems C and D]{Ang88} to complete the proof.\footnote{Note that the argument of the Dirichlet case of \cite[Theorems C]{Ang88} yields the same conclusions under the mixed boundary condition: $u(x_1,t)=0$ and $u(x_2,t)\ne 0$ for all $t$.}
\end{proof}



Denoting the total curvature of $\Gamma_t$ by
\[
K(\Gamma_t)\doteqdot \int_{\Gamma_t}\vert\kappa\vert\,,
\]
we now obtain the following free boundary version of Altschuler's formula.

\begin{lemma}\label{lem:Altschuler}
On any compact free boundary curve shortening flow $\{\Gamma_t\}_{t\in(\alpha,\omega)}$ in a convex domain $\Omega\subset \R^2$, we have
\begin{equation}\label{eq:Altschuler}
\frac{d}{dt}K(\Gamma_t)=
\sum_{\pd\Gamma_t}\vert\kappa\vert\kappa^S-2\sum_{\{p:\kappa(p,\cdot)=0\}}\vert\cd\kappa\vert\,,
\end{equation}
except at finitely many times.
\end{lemma}
\begin{proof}
By 
Lemma \ref{lem:Sturm}, either the solution is a stationary chord or shrinking semicircle (and hence the claim holds trivially) or the inflection points of $\Gamma_t$ are finite in number and non-degenerate, except possibly at a finite set of times
. Away from these times, we may split $\Gamma_t$ into $N$ segments $\{\Gamma_t^j\}_{j=1}^N$, with boundaries $\{a_{i-1},a_i\}_{i=1}^N$, on which $\kappa$ is nonzero and alternates sign, so that, for an appropriate choice of arclength parameter,
\bann
\frac{d}{dt}\int_{\Gamma_t}\vert\kappa\vert={}&\sum_{j=1}^N(-1)^{j-1}\frac{d}{dt}\int_{\Gamma_t^j}\kappa\,ds\\
={}&\sum_{i=1}^N(-1)^{j-1}\int_{\Gamma_t^j}\kappa_{ss}\,ds\\
={}&-\kappa_s(a_0)+2\sum_{j=1}^{N-1}(-1)^{j-1}\kappa_{s}(a_j)+(-1)^{N-1}\kappa_s(a_N)\,.
\eann
Observe that $(-1)^{i}\kappa_{s}(a_i)\ge 0$ for each $i$ and
\ba\label{eq:boundary gradient}
0={}&\frac{d}{dt}\inner{N}{N^S}\nonumber\\
={}&\inner{\cd\kappa}{N^S}-\kappa\inner{N}{D_NN^S}\nonumber\\
={}&\inner{\cd\kappa}{N^S}-\kappa\kappa^S
\ea
at the boundary. The claim follows since $\inner{\pd_s}{N^S}=-1$ at $a_0$ and $+1$ at $a_N$.
\end{proof}

We may eliminate the boundary term in \eqref{eq:Altschuler}, resulting in a genuine monotonicity formula, by introducing the total curvature $\tilde K(\Gamma_t)$ of the portion of the boundary $\pd\Omega$ (counted with multiplicity) traversed by the endpoints of $\Gamma_t$. That is, we set
\[
\tilde K(\Gamma_t)
\doteqdot \int_t^T\vert\kappa^\text{L}\vert\,ds_{\mathrm{L}}+\int_t^T\vert\kappa^\text{R}\vert\,ds_{\mathrm{R}}\,,
\]
where $\kappa^{\text{L}}$ (resp. $\kappa^{\text{R}}$) denotes the curvature and $ds_{\mathrm{L}}$ (resp. $ds_{\mathrm{R}}$) the length element of the piecewise smoothly immersed curve $\zeta^{\text{L}}:[0,T)\to\pd\Omega$ (resp. $\zeta^{\text{R}}:[0,T)\to\pd\Omega$) determined by the left (resp. right) boundary point $\zeta_{\text{L}}(t)$ (resp. $\zeta_{\text{R}}(t)$) of $\pd\Gamma_t$. 

Recall that the boundary points may only change direction at most finitely many times. Moreover, the boundary avoidance estimate, Proposition \ref{prop:boundary-avoidance}, provides room for uniform barriers that prevent the boundary from cycling $\pr\Omega$ an infinite number of times as $t\to T<\infty$. It follows that the boundary points converge, and in particular $\tilde K(\Gamma_t)$ is finite.

Now observe that, away from the (finitely many) boundary inflection times, the rate of change of $\tilde K(\Gamma_t)$ exactly cancels the boundary term in \eqref{eq:Altschuler}. Thus, if we define
\[
\bs K(\Gamma_t)\doteqdot K(\Gamma_t)+\tilde K(\Gamma_t)\,,
\]
then we obtain the following monotonicity formula.
\begin{corollary}
On any compact free boundary curve shortening flow $\{\Gamma_t\}_{t\in(\alpha,\omega)}$ in a convex domain $\Omega\subset \R^2$
\begin{equation}\label{eq:Altschuler_prime}
\frac{d}{dt}\bs K(\Gamma_t)=-2\sum_{\{p:\kappa(p,\cdot)=0\}}\vert\cd\kappa\vert
\end{equation}
except at finitely many times.
\end{corollary}


Putting these ingredients together, we arrive at Theorem \ref{thm:fbGrayson}.

\begin{proof}[Proof of Theorem \ref{thm:fbGrayson} part (b) (intrinsic method)]
By hypothesis, $T<\infty$. By applying the ODE comparison principle to \eqref{eq:evolvekappa}, we find that
\begin{equation}\label{eq:kappa lower bound}
\max_{\Gamma_t}\kappa^2\ge \frac{1}{2(T-t)}\,.
\end{equation}
We claim that
\begin{equation}\label{eq:type I}
\limsup_{t\to T}\max_{\Gamma_t}(T-t)\kappa^2<\infty\,.
\end{equation}
Indeed, if this is not the case, then we may blow-up \emph{\`a la Hamilton} to obtain a Grim Reaper solution, which will contradict the chord-arc estimate: choose a sequence of times $t_{j}\in[0,T-j^{-1})$ and a sequence of points $x_{j}\in M$ such that
\[
\kappa^{2}(x_{j},t_{j})\left(T-\tfrac{1}{j}-t_{j}\right)=\max_{(x,t)\in M\times \left [0,T-\frac{1}{j}\right]}\kappa^{2}(x,t)\left(T-\tfrac{1}{j}-t\right)\,,
\]
set $\sigma_j\doteqdot\lambda_j^2t_j$ and $T_{j}=\lambda^2_{j}\left(T-\frac{1}{j}-t_{j}\right)$, where $\lambda_j\doteqdot \vert\kappa(x_{j},t_{j})\vert$, and consider the sequence of rescaled solutions $\gamma_j:M\times (-\sigma_j,T_j)\to \Omega_j\doteqdot \lambda_j(\Omega-\gamma(x_j,t_j))$ defined by
\[
\gamma_{j}(x,t)=\lambda_{j}\left(\gamma\left(x,t_{j}+\lambda_j^{-2}t\right)-\gamma(x_{j},t_{j})\right).
\]
By hypothesis, we can pass to some subsequence such that $t_j\to T$, $\lambda_j\to\infty$, $\sigma_j\to \infty$, $T_j\to\infty$, and $\Omega_j\to\Omega_\infty$, where $\Omega_\infty$ is either the plane or some halfplane. Observe also that $\gamma_{j}(x_{j},0)=0$ and
\begin{align*}
\kappa_{j}^2(x,t)={}&\lambda_j^{-2}\kappa^2\left(x,t_{j}+\lambda_j^{-2}t\right)\leq\frac{T-\frac{1}{j}-t_{j}}{T-\frac{1}{j}-(t_{j}+\lambda_j^{-2}t)}=\frac{T_{j}}{T_{j}-t}.
\end{align*}
So the curvature is uniformly bounded on any compact time interval and, up to a change of orientation, takes unit value at the spacetime origin. Thus, by estimates for quasilinear parabolic partial differential equations with transverse Neumann boundary condition (cf. \cite{Stahl96b}), the rescaled solutions $\gamma_j:M\times [-\sigma_j,T_j)\to\Omega_j$ converge in $C^\infty$, after passing to a subsequence, to a smooth, proper limiting solution $\gamma_{\infty}:M_\infty\times(-\infty,\infty)\to \Omega_\infty$ uniformly on compact subsets of $M_\infty\times(-\infty,\infty)$. 

By an argument of Altschuler \cite{Altschuler}, we find that the limit flow is locally uniformly convex (and hence convex by the strong maximum principle and the curvature normalization at the spacetime origin): 
integrating the identity \eqref{eq:Altschuler_prime} in time on the rescaled flows between times $a$ and $b$ yields
\[
\bs K(\Gamma_{\lambda_j^{-2}b+t_j})-\bs{K}(\Gamma_{\lambda_j^{-2}a+t_j})=-2\sum_{\{p:\kappa^j(p,t)=0\}}\int_a^b\vert\cd\kappa^j\vert dt\,.
\]
Since $\bs{K}(\Gamma_t)$ is non-negative and non-increasing, it takes a limit as $t\to T$, and hence the left hand side tends to zero as $j\to\infty$. Since $a$ and $b$ were arbitrary, we conclude that any inflection point of the limit flow is degenerate. Since the limit flow is not a critical chord (due to the normalization of $\kappa$ at the spacetime origin), we conclude that there are no inflection points, and hence the limit is indeed locally uniformly convex. Proposition \ref{prop:ancient} now implies that the limit solution, being eternal and non-flat, is a (half-)Grim Reaper, which is impossible due to the chord-arc estimate. This proves \eqref{eq:type I}.

We next claim that
\ba\label{eq:smoothestimatesHamilton}
\limsup_{t\nearrow T}\max_{\Gamma_t}\vert 2(T-t)\kappa^2-1\vert=0\;\;\text{and}\;\;\limsup_{t\nearrow T}\max_{\Gamma_t}(T-t)^{-(m+1)}\vert \cd^m\kappa\vert^2=0
\ea
for all $m\in\N$. 
Indeed, given any sequence of times $t_j\to T$, choose $x_j\in M$ so that $\kappa^2(x_j,t_j)=\max_{x\in M}\kappa^2(x,t_j)$, set $\lambda_j\doteqdot (T-t_j)^{-\frac{1}{2}}$ and $\sigma_j\doteqdot\lambda_j^2t_j$, and consider the sequence of rescaled solutions $\gamma_j:M\times (-\sigma_j,1)\to\Omega_j\doteqdot\lambda_j(\Omega-\gamma(x_j,t_j))$ defined by
\[
\gamma_{j}(\,\cdot\,,t)=\lambda_{j}\left(\gamma\left(\,\cdot\,,t_{j}+\lambda_j^{-2}t\right)-\gamma(x_{j},t_{j})\right)\,.
\]
For each $j$, $\gamma_{j}(x_{j},0)=0$, $\sigma_j\to\infty$, and
\[
\kappa_j^2(\,\cdot\,,t)=\lambda_j^{-2}\kappa^2\left(\,\cdot\,,t_j+\lambda_j^{-2}t\right)\leq \frac{C^2}{1-t}\,.
\]
Thus, as above, 
the rescaled solutions $\gamma_j:M\times [-\sigma_j,1)\to\Omega_j$ converge in the smooth topology, after passing to a subsequence, to a smooth, proper limiting solution $\gamma_{\infty}:M_\infty\times(-\infty,1)\to\Omega_\infty$ uniformly on compact subsets of $M_\infty\times(-\infty,1)$, where $\Omega_\infty$ is either the plane or a halfplane. 
By \eqref{eq:kappa lower bound}, the curvature must, up to a change of orientation, be positive at the spacetime origin, and hence positive everywhere by the above argument. Since the chord-arc estimate is scale invariant, we deduce from Proposition \ref{prop:ancient} that the limit is a shrinking semicircle, from which the estimates \eqref{eq:smoothestimatesHamilton} follow.

Convergence to a point on $\pd\Omega$ now follows by integrating the curve shortening flow equation and applying the first of the estimates \eqref{eq:smoothestimatesHamilton}; smooth convergence to the corresponding unit semi-circle after rescaling then follows by converting the geometric estimates \eqref{eq:smoothestimatesHamilton} into estimates for the rescaled immersions.
\end{proof}

\subsection{Remarks}

\emph{Existence} of a (geometrically unique) free boundary curve shortening flow out of any given embedded closed interval having orthogonal boundary condition in a convex domain was proved by Stahl 
\cite{Stahl96a}.

Note that our argument for part (b) of Theorem \ref{thm:fbGrayson} does not require Stahl's result \cite[Proposition 1.4]{Stahl96b} on the convergence to points of bounded, convex, non-flat free boundary curve shortening flows in convex domains, and hence provides a new proof of it (finiteness of $T$ is a straightforward consequence of non-trivial convexity and the maximum principle; cf. \cite[Theorem 3.2]{Stahl96b}).

We found it convenient in the intrinsic blow-up approach to exploit the full classification of convex ancient solutions to planar curve shortening flow \cite{BLTcsf} (via Proposition \ref{prop:ancient}). It would suffice, as in \cite{Huisken96}, to exploit the (easier) classification of convex translators and shrinkers, however, since a suitable monotonicity formula is available \cite{Buckland} (see also \cite[Theorem 5.1]{Edelen}) and the differential Harnack inequality (which is not available in the general free boundary setting) may be invoked \emph{after} obtaining an eternal convex limit flow in $\R^2_+$ or $\R^2$.

\section{Remarks on unbounded solutions}
\label{sec:unbounded}

The above arguments also yield information for solutions with unbounded timeslices (in unbounded convex domains $\Omega$) so long as suitable conditions at infinity are in place.

In this case, since $L=\infty$, we consider the unnormalized auxiliary functions
\[ Z(x,y,t) = d(\gamma(x,t),\gamma(y,t))-\phi\left(\ell(x,y,t),t\right), \]
\[ \tilde{Z}(x,y,t) = \tilde{d}(\gamma(x,t),\gamma(y,t),t)-\phi(\tilde{\ell}(x,y,t),t),\]
\[\bar{Z}(x,y,z,t) = d(\gamma(x,t) ,z) + d(\gamma(y,t) ,z)-\phi(\tilde{\ell}(x,y,t ),t)\]
and
\[
\bs{Z}(x,y,t) \doteqdot \bs{d}(x,y,t)-\phi\left(\bs{\ell}(x,y,t),t\right)\,,
\]
where $\phi$ is any smooth modulus. Following Sections \ref{sec:spatial variation} and \ref{sec:evolution}, we find (in particular) that 
\bann
\phi_t
>{}& 4\phi''
\eann
at an interior parabolic minimum of $\bs Z$. 

Taking $\phi(\lambda,t)=\varepsilon\lambda$ generalizes an estimate of Huisken \cite[Theorem 2.1]{Huisken96}: if each timeslice $\Gamma_t$ has one end and the ends are all asymptotic to a fixed ray, then initial lower bounds for $\bs d/\bs\ell$ are preserved. 
If the asymptotic ray points into the interior of the asymptotic cone of $\Omega$, then we find that $\bs d/\bs\ell$ is uniformly bounded from below. This rules out finite time singularities and, we expect, can be used to give a simple proof of smooth convergence to a half-line or a half-expander as $t\to\infty$.

Note that the above estimate is vacuous if the asymptotic ray does not point into the interior of the asymptotic cone of $\Omega$ (since in that case $\inf_{\Gamma_0}\bs d/\bs\ell\equiv 0$). On the other hand, taking $\phi$ to be the error function solution to the heat equation, we may still obtain an (exponentially decaying) lower bound of the form $\bs d\ge \phi(\bs\ell,t)$ in the case that the asymptotic ray is parallel to the boundary of the asymptotic cone
. This also rules out finite time singularities and, we expect, can be used to give a simple proof of smooth convergence to a half-line or a half-Grim Reaper as $t\to\infty$.

\bibliographystyle{acm}
\bibliography{bibliography}

\begin{thebibliography}{10}

\bibitem{Altschuler}
{\sc Altschuler, S.~J.}
\newblock Singularities of the curve shrinking flow for space curves.
\newblock {\em J. Differential Geom. 34}, 2 (1991), 491--514.

\bibitem{AltschulerWu}
{\sc {Altschuler}, S.~J., and {Wu}, L.~F.}
\newblock {Translating surfaces of the non-parametric mean curvature flow with
  prescribed contact angle.}
\newblock {\em {Calc. Var. Partial Differ. Equ.} 2}, 1 (1994), 101--111.

\bibitem{MR2967056}
{\sc Andrews, B.}
\newblock Noncollapsing in mean-convex mean curvature flow.
\newblock {\em Geom. Topol. 16}, 3 (2012), 1413--1418.

\bibitem{AndrewsSurvey}
{\sc Andrews, B.}
\newblock Moduli of continuity, isoperimetric profiles, and multi-point
  estimates in geometric heat equations.
\newblock In {\em Surveys in differential geometry 2014. {R}egularity and
  evolution of nonlinear equations}, vol.~19 of {\em Surv. Differ. Geom.} Int.
  Press, Somerville, MA, 2015, pp.~1--47.

\bibitem{MR2843240}
{\sc Andrews, B., and Bryan, P.}
\newblock A comparison theorem for the isoperimetric profile under
  curve-shortening flow.
\newblock {\em Comm. Anal. Geom. 19}, 3 (2011), 503--539.

\bibitem{AndrewsBryan}
{\sc Andrews, B., and Bryan, P.}
\newblock Curvature bound for curve shortening flow via distance comparison and
  a direct proof of {G}rayson's theorem.
\newblock {\em J. Reine Angew. Math. 653\/} (2011), 179--187.

\bibitem{EGF}
{\sc Andrews, B., Chow, B., Guenther, C., and Langford, M.}
\newblock {\em Extrinsic {G}eometric {F}lows}, first~ed., vol.~206 of {\em
  Graduate {S}tudies in {M}athematics}.
\newblock American {M}athematical {S}ociety, {P}rovidence, {RI}, 2020.

\bibitem{Ang88}
{\sc {Angenent}, S.}
\newblock {The zero set of a solution of a parabolic equation.}
\newblock {\em {J. Reine Angew. Math.} 390\/} (1988), 79--96.

\bibitem{Bakas}
{\sc Bakas, I., and Sourdis, C.}
\newblock Dirichlet sigma models and mean curvature flow.
\newblock {\em Journal of High Energy Physics 2007}, 06 (2007), 057.

\bibitem{BLTcsf}
{\sc Bourni, T., Langford, M., and Tinaglia, G.}
\newblock Convex ancient solutions to curve shortening flow.
\newblock {\em Calc. Var. Partial Differential Equations 59}, 4 (2020), 133.

\bibitem{MR3061135}
{\sc Brendle, S.}
\newblock Minimal surfaces in {$S^3$}: a survey of recent results.
\newblock {\em Bull. Math. Sci. 3}, 1 (2013), 133--171.

\bibitem{Buckland}
{\sc Buckland, J.~A.}
\newblock Mean curvature flow with free boundary on smooth hypersurfaces.
\newblock {\em J. Reine Angew. Math. 586\/} (2005), 71--90.

\bibitem{MR1669736}
{\sc Chopard, B., and Droz, M.}
\newblock {\em Cellular automata modeling of physical systems}.
\newblock Collection Al\'{e}a-Saclay: Monographs and Texts in Statistical
  Physics. Cambridge University Press, Cambridge, 1998.

\bibitem{DHScsf}
{\sc Daskalopoulos, P., Hamilton, R., and Sesum, N.}
\newblock Classification of compact ancient solutions to the curve shortening
  flow.
\newblock {\em J. Differential Geom. 84}, 3 (2010), 455--464.

\bibitem{WhiteNotesEdelen}
{\sc Edelen, N.}
\newblock Notes from {B}rian {W}hite's class on mean curvature flow.
\newblock Unpublished notes,
  \href{https://nedelen.science.nd.edu/brian-mcf-notes.pdf}{https://nedelen.science.nd.edu/brian-mcf-notes.pdf}
  (accessed 27 Jan. 2023).

\bibitem{Edelen}
{\sc Edelen, N.}
\newblock The free-boundary {B}rakke flow.
\newblock {\em J. Reine Angew. Math. 758\/} (2020), 95--137.

\bibitem{EHIZ}
{\sc Edelen, N., Haslhofer, R., Ivaki, M.~N., and Zhu, J.~J.}
\newblock Mean convex mean curvature flow with free boundary.
\newblock {\em Comm. Pure Appl. Math. 75}, 4 (2022), 767--817.

\bibitem{Firey}
{\sc Firey, W.~J.}
\newblock Shapes of worn stones.
\newblock {\em Mathematika 21\/} (1974), 1--11.

\bibitem{GageHamilton86}
{\sc {Gage}, M., and {Hamilton}, R.}
\newblock {The heat equation shrinking convex plane curves.}
\newblock {\em {J. Differ. Geom.} 23\/} (1986), 69--96.

\bibitem{Gage84}
{\sc Gage, M.~E.}
\newblock Curve shortening makes convex curves circular.
\newblock {\em Invent. Math. 76}, 2 (1984), 357--364.

\bibitem{MR1046497}
{\sc Gage, M.~E.}
\newblock Curve shortening on surfaces.
\newblock {\em Ann. Sci. \'{E}cole Norm. Sup. (4) 23}, 2 (1990), 229--256.

\bibitem{GilbargTrudinger}
{\sc {Gilbarg}, D., and {Trudinger}, N.~S.}
\newblock {\em {Elliptic partial differential equations of second order.
  Reprint of the 1998 ed.}}, reprint of the 1998 ed.~ed.
\newblock Berlin: Springer, 2001.

\bibitem{Grayson}
{\sc Grayson, M.~A.}
\newblock The heat equation shrinks embedded plane curves to round points.
\newblock {\em J. Differential Geom. 26}, 2 (1987), 285--314.

\bibitem{MR979601}
{\sc Grayson, M.~A.}
\newblock Shortening embedded curves.
\newblock {\em Ann. of Math. (2) 129}, 1 (1989), 71--111.

\bibitem{MR2931330}
{\sc Halldorsson, H.~P.}
\newblock Self-similar solutions to the curve shortening flow.
\newblock {\em Trans. Amer. Math. Soc. 364}, 10 (2012), 5285--5309.

\bibitem{HamiltonPinched}
{\sc Hamilton, R.~S.}
\newblock Convex hypersurfaces with pinched second fundamental form.
\newblock {\em Comm. Anal. Geom. 2}, 1 (1994), 167--172.

\bibitem{MR1369140}
{\sc Hamilton, R.~S.}
\newblock Isoperimetric estimates for the curve shrinking flow in the plane.
\newblock In {\em Modern methods in complex analysis ({P}rinceton, {NJ},
  1992)}, vol.~137 of {\em Ann. of Math. Stud.} Princeton Univ. Press,
  Princeton, NJ, 1995, pp.~201--222.

\bibitem{Huisken89}
{\sc Huisken, G.}
\newblock Nonparametric mean curvature evolution with boundary conditions.
\newblock {\em J. Differential Equations 77}, 2 (1989), 369--378.

\bibitem{Huisken96}
{\sc Huisken, G.}
\newblock A distance comparison principle for evolving curves.
\newblock {\em Asian J. Math. 2}, 1 (1998), 127--133.

\bibitem{Lieberman}
{\sc Lieberman, G.~M.}
\newblock {\em Second order parabolic differential equations}.
\newblock World Scientific Publishing Co., Inc., River Edge, NJ, 1996.

\bibitem{Mullins}
{\sc Mullins, W.~W.}
\newblock Two-dimensional motion of idealized grain boundaries.
\newblock {\em J. Appl. Phys. 27\/} (1956), 900--904.

\bibitem{Sapiro}
{\sc Sapiro, G.}
\newblock {\em Geometric partial differential equations and image analysis}.
\newblock Cambridge University Press, Cambridge, 2001.

\bibitem{Schulze}
{\sc Schulze, F.}
\newblock Introduction to mean curvature flow.
\newblock Unpublished lecture notes
  \href{https:/www.felixschulze.eu/images/mcf_notes.pdf}{https:/www.felixschulze.eu/images/mcf{\textunderscore}notes.pdf}
  (accessed 27 Jan. 2023).

\bibitem{Stahl96b}
{\sc Stahl, A.}
\newblock Convergence of solutions to the mean curvature flow with a {N}eumann
  boundary condition.
\newblock {\em Calc. Var. Partial Differential Equations 4}, 5 (1996),
  421--441.

\bibitem{Stahl96a}
{\sc Stahl, A.}
\newblock Regularity estimates for solutions to the mean curvature flow with a
  {N}eumann boundary condition.
\newblock {\em Calc. Var. Partial Differential Equations 4}, 4 (1996),
  385--407.

\bibitem{vonNeumann}
{\sc von Neumann, J.}
\newblock In {\em Metal Interfaces}. American Society for Testing Materials,
  Cleveland, 1952, p.~108.

\bibitem{Wh03}
{\sc White, B.}
\newblock The nature of singularities in mean curvature flow of mean-convex
  sets.
\newblock {\em J. Amer. Math. Soc. 16}, 1 (2003), 123--138 (electronic).

\end{thebibliography}

\end{document}